\documentclass[10pt,11pt]{amsart}%
\usepackage[dvips]{graphicx}
\usepackage{amsmath}
\usepackage{color}
\usepackage{multicol}
\usepackage{amsfonts}
\usepackage{amssymb}
\usepackage{version}%
\providecommand{\U}[1]{\protect\rule{.1in}{.1in}}
\providecommand{\U}[1]{\protect\rule{.1in}{.1in}}
\providecommand{\U}[1]{\protect\rule{.1in}{.1in}}
\providecommand{\U}[1]{\protect\rule{.1in}{.1in}}
\providecommand{\U}[1]{\protect\rule{.1in}{.1in}}
\providecommand{\U}[1]{\protect\rule{.1in}{.1in}}
\providecommand{\U}[1]{\protect\rule{.1in}{.1in}}
\newtheorem{theorem}{Theorem}[section]

\newtheorem{corollary}[theorem]{Corollary}

\newtheorem{definition}[theorem]{Definition}
\newtheorem{example}[theorem]{Example}

\newtheorem{lemma}[theorem]{Lemma}

\newtheorem{proposition}[theorem]{Proposition}
\newtheorem{remark}[theorem]{Remark}
\newtheorem{remarks}[theorem]{Remarks}

\begin{document}
\title[Central measures and Littelmann paths]{Harmonic functions on multiplicative graphs and inverse Pitman transform on
infinite random paths}
\date{February, 2015}
\author{C\'{e}dric Lecouvey, Emmanuel Lesigne and Marc Peign\'{e}}
\maketitle

\begin{abstract}
We introduce and characterize central probability distributions on Littelmann
paths. Next we establish a law of large numbers and a central limit theorem
for the generalized Pitmann transform. We then study harmonic functions on
multiplicative graphs defined from the tensor powers of finite-dimensional Lie
algebras representations. Finally, we show there exists an inverse of the
generalized Pitman transform defined almost surely on the set of infinite
paths remaining in the Weyl chamber and explain how it can be computed.

\end{abstract}

\section{Introduction}

In this paper we apply algebraic and combinatorial tools coming from
representation theory of Lie algebras to the study of random paths.\ In
\cite{LLP} and \cite{LLP3} we investigate the random Littelmann path defined
from a simple module $V$ of a Kac-Moody algebra $\mathfrak{g}$ and use the
generalized Pitmann transform $\mathcal{P}$ introduced by Biane, Bougerol and
O'Connell \cite{BBO} to obtain its conditioning to stay in the dominant Weyl
chamber of $\mathfrak{g}$.\ Roughly speaking, this random path is obtained by
concatenation of elementary paths randomly chosen among the vertices of the
crystal graph $B$ associated to $V$ following a distribution depending on the
graph structure of $B$.\ It is worth noticing that for $\mathfrak{g=sl}_{2}$,
this random path reduces to the random walk on $\mathbb{Z}$ with steps
$\{\pm1\}$ and the transform $\mathcal{P}$ is the usual Pitman transform
\cite{Pit}.\ Also when $V$ is the defining representation of $\mathfrak{g}%
=\mathfrak{sl}_{n+1}$, the vertices of $B$ are simply the paths linking $0$ to
each vector of the standard basis of $\mathbb{R}^{n+1}$ and we notably
recovered some results by O'Connell exposed in \cite{OC1}. It appears that
many natural random walks can in fact be realized from a suitable choice of
the representation $V$.

We will assume here that $\mathfrak{g}$ is a simple (finite-dimensional) Lie
algebra over $\mathbb{C}$ of rank $n$.\ The irreducible finite-dimensional
representations of $\mathfrak{g}$ are then parametrized by the dominant
weights of $\mathfrak{g}$ which are the elements of the set $P_{+}%
=P\cap\mathcal{C}$ where $P$ and $\mathcal{C}$ are the weight lattice and the
dominant Weyl chamber of $\mathfrak{g}$, respectively. The random path
$\mathcal{W}$ we considered in \cite{LLP3} is defined from the crystal
$B(\kappa)$ of the irreducible $\mathfrak{g}$-module $V(\kappa)$ with highest
weight $\kappa\in P_{+}$ ($\kappa$ is fixed for each $\mathcal{W}$). The
crystal $B(\kappa)$ is an oriented graph graded by the weights of
$\mathfrak{g}$ whose vertices are Littelmann paths of length $1$. The vertices
and the arrows of $B(\kappa)$ are obtained by simple combinatorial rules from
a path $\pi_{\kappa}$ connecting $0$ to $\kappa$ and remaining in
$\mathcal{C}$ (the highest weight path). We endowed $B(\kappa)$ with a
probability distribution $p$ compatible with the weight graduation defined
from the choice of a $n$-tuple $\tau$ of positive reals (a positive real for
each simple root of $\mathfrak{g}$).\ The probability distribution considered
on the successive tensor powers $B(\kappa)^{\otimes\ell}$ is the product
distribution $p^{\otimes\ell}$.\ It has the crucial property to be central:
two paths in $B(\kappa)^{\otimes\ell}$ with the same ends have the same
probability. We can then define, following the classical construction of a
Bernoulli process, a random path $\mathcal{W}$ with underlying probability
space $(B(\kappa)^{\otimes\mathbb{Z}_{\geq0}},p^{\otimes\mathbb{Z}_{\geq0}})$
as the direct limit of the spaces $(B(\kappa)^{\otimes\ell},p^{\otimes\ell}%
)$.\ The trajectories of $\mathcal{W}$ are the concatenations of the
Littelmann paths appearing in $B(\kappa)$.\ It makes sense to consider the
image of $\mathcal{W}$ by the generalized Pitman transform $\mathcal{P}$. This
yields a Markov process $\mathcal{H=P(W)}$ whose trajectories are the
concatenations of the paths appearing in $B(\kappa)$ which remain in the
dominant Weyl chamber $\mathcal{C}$.\ When the drift of $\mathcal{W}$ belongs
to the interior of $\mathcal{C}$, we establish in \cite{LLP3} that the law of
$\mathcal{H}$ coincides with the law of $\mathcal{W}$ conditioned to stay in
$\mathcal{C}$.\ By setting $W_{\ell}=\mathcal{W}(\ell)$ for any positive
integer $\ell$, we obtain in particular a Markov chain $W=(W_{\ell})_{\ell
\geq1}$ on the dominant weights of $\mathfrak{g}$.

In the spirit of the works of Kerov and Vershik, one can define central
probability measures on the space $\Omega_{\mathcal{C}}$ of infinite
trajectories associated to $\mathcal{H}$ (i.e. remaining in $\mathcal{P}%
$).\ These are the probability measures giving the same probability to any
cylinders $C_{\pi}$ and $C_{\pi^{\prime}}$ issued from paths $\pi$ and
$\pi^{\prime}$ of length $\ell$ remaining in $\mathcal{C}$ with the same ends.
Alternatively, we can consider the multiplicative graph $\mathcal{G}$ with
vertices the pairs $(\lambda,\ell)\in P_{+}\times\mathbb{Z}_{\geq0}$ and
weighted arrows $(\lambda,\ell)\overset{m_{\lambda,\kappa}^{\Lambda}%
}{\rightarrow}(\Lambda,\ell+1)$ where $m_{\lambda,\kappa}^{\Lambda}$ is the
multiplicity of the representation $V(\Lambda)$ in the tensor product
$V(\lambda)\otimes V(\kappa)$.\ Each central probability measure on
$\Omega_{\mathcal{C}}$ is then characterized by the harmonic function
$\varphi$ on $\mathcal{G}$ associating to each vertex $(\lambda,\ell)$, the
probability of any cylinder $C_{\pi}$ where $\pi$ is any path of length $\ell$
remaining in $\mathcal{C}$ and ending at $\lambda$. Finally, a third
equivalent way to study central probability measures on $\Omega_{\mathcal{C}}$
is to define a Markov chain on $\mathcal{G}$ whose transition matrix is
computed from the harmonic function $\varphi.$ We refer to Paragraph
\ref{subsec_graphs} for a detailed review.

When $\mathfrak{g}=\mathfrak{sl}_{n+1}$, the elements of $P_{+}$ can be
regarded as the partitions $\lambda=(\lambda_{1}\geq\cdots\geq\lambda_{n}%
\geq0)\in\mathbb{Z}^{n}$. Moreover, if we choose $V(\kappa)=V$, the defining
representation of $\mathfrak{g}=\mathfrak{sl}_{n+1}$, we have $m_{\lambda
,\kappa}^{\Lambda}\neq0$ if and only if the Young diagram of $\Lambda$ is
obtained by adding one box to that of $\lambda$.\ The connected component of
$\mathcal{G}$ obtained from $(\emptyset,0)$ thus coincides with the Young
lattice $\mathcal{Y}_{n}$ of partitions with at most $n$ parts (one can obtain
the whole Young lattice $\mathcal{Y}$ by working with $\mathfrak{g}%
=\mathfrak{sl}_{\infty}$).\ In that case, Kerov and Vershik (see \cite{Ker})
completely determined the harmonic function on $\mathcal{Y}$.\ They showed
that these harmonic functions have nice expressions in terms of generalized
Schur functions.

In \cite{Pit} Pitman established that the usual (one-dimensional) Pitman
transform is almost surely invertible on infinite trajectories (i.e.
reversible on a space of trajectories of probability $1$). It is then a
natural question to ask wether its generalized version $\mathcal{P}$ shares
the same invertibility property. Observe that in the case of the defining
representation of $\mathfrak{sl}_{n+1}$ (or $\mathfrak{sl}_{\infty}$), the
generalized Pitmann transform can be expressed in terms of a
Robinson-Schensted-Knuth (RSK) type correspondence.\ Such an invertibility
property was obtained by O'Connell in \cite{OC1} (for usual RSK related to
ordinary Schur functions) and recently extended by Sniady \cite{Snia} (for the
generalized version of RSK used by Kerov and Vershik and related to the
generalized Schur functions).{ Our result shows that this invertibility
property survives beyond type }$A$ and for random paths constructed from any
irreducible representation.

\bigskip

In what follows, we first prove that the probability distributions $p$ on
$B(\kappa)$ we introduced in \cite{LLP}, \cite{LLP2} and \cite{LLP3} are
precisely all the possible distributions yielding central distributions on
$B(\kappa)^{\otimes\ell}$. We believe this will make the restriction we did in
these papers more natural. We also establish a law of large numbers and a
central limit theorem for the Markov process $\mathcal{H}$.\ Here we need our
assumption that $\mathfrak{g}$ is finite-dimensional since in this case
$\mathcal{P}$ has a particular simple expression as a composition of
(ordinary) Pitman transforms. Then we determine the harmonic functions on the
multiplicative graph $\mathcal{G}$ for which the associated Markov chain
verifies a law of large numbers. We establish in fact that these Markov chains
are exactly the processes $H$ defined in \cite{LLP} and have simple
expressions in terms of the Weyl characters of $\mathfrak{g}$. This can be
regarded as an analogue of the result of Kerov and Vershik determining the
harmonic functions on the Young lattice.\ Finally, we prove that the
generalized Pitman transform $\mathcal{P}$ is almost surely invertible and
explain how its inverse can be computed. Here we will extend the approach
developed by Sniady in \cite{Snia} for the generalized RSK to our context.

\bigskip

The paper is organized as follows. In Section 2, we recall some background on
continuous time Markov processes. Section 3 is a recollection of results on
representation theory of Lie algebras and the Littelmann path model.\ We state
in Section 4 the main results of \cite{LLP3} and prove that the probability
distributions $p$ introduced in \cite{LLP} are in fact the only possible
yielding central measures on trajectories. The law of large numbers and the
central limit theorem for $\mathcal{H}$ are established in Section 5. We study
the harmonic functions of the graphs $\mathcal{G}$ in Section 6.\ In Section 7
we show that the spaces of trajectories for $\mathcal{W}$ and $\mathcal{H}$
both have the structure of dynamical systems coming from the shift
operation.\ We then prove that these dynamical systems are intertwined by
$\mathcal{P}$. Finally, we establish the existence of a relevant inverse of
$\mathcal{P}$ in Section 7.

\bigskip

\noindent\textbf{MSC classification:} 05E05, 05E10, 60G50, 60J10, 60J22.

\section{Random paths}

\subsection{Background on Markov chains}

\label{subsec-Markov} Consider a probability space $(\Omega,{\mathcal{F}%
},{\mathbb{P}})$ and a countable set $M$. A sequence $Y=(Y_{\ell})_{\ell\geq
0}$ of random variables defined on $\Omega$ with values in $M$ is a
\textit{Markov chain} when
\[
{\mathbb{P}}(Y_{\ell+1}=\mu_{\ell+1}\mid Y_{\ell}=\mu_{\ell},\ldots,Y_{0}%
=\mu_{0})={\mathbb{P}}(Y_{\ell+1}=\mu_{\ell+1}\mid Y_{\ell}=\mu_{\ell})
\]
for any any $\ell\geq0$ and any $\mu_{0},\ldots,\mu_{\ell},\mu_{\ell+1}\in M$.
The Markov chains considered in the sequel will also be assumed time
homogeneous, that is ${\mathbb{P}}(Y_{\ell+1}=\lambda\mid Y_{\ell}%
=\mu)={\mathbb{P}}(Y_{\ell}=\lambda\mid Y_{\ell-1}=\mu)$ for any $\ell\geq1$
and $\mu,\lambda\in M$.\ For all $\mu,\lambda$ in $M$, the transition
probability from $\mu$ to $\lambda$ is then defined by
\[
\Pi(\mu,\lambda)={\mathbb{P}}(Y_{\ell+1}=\lambda\mid Y_{\ell}=\mu)
\]
and we refer to $\Pi$ as the transition matrix of the Markov chain $Y$. The
distribution of $Y_{0}$ is called the initial distribution of the chain $Y$.

{\textit{A continuous time Markov process} ${\mathcal{Y}}=({\mathcal{Y}%
}(t))_{t\geq0}$ on $(\Omega,{\mathcal{F}},{\mathbb{P}})$ with values in
${\mathbb{R}}^{n}$ is a measurable family of random variables defined on
$(\Omega,{\mathcal{F}},{\mathbb{P}})$ such that, for any integer $k\geq1$ and
any $0\leq t_{1}<\cdots<t_{k+1}$ the conditional distribution $^{(}%
$\footnote{Let us recall briefly the definition of the conditional
distribution of a random variable given another one. Let $X$ and $Y$ be random
variables defined on some probability space $(\Omega,\mathcal{F},\mathbb{P})$
with values respectively in $\mathbb{R}^{n}$ and $\mathbb{R}^{m},n,m\geq1$.
Denote by $\mu_{X}$ the distribution of $X$, it is a probability measure on
$\mathbb{R}^{n}$. The conditional distribution of $Y$ given $X$ is defined by
the following \textquotedblleft disintegration\textquotedblright\ formula: for
any Borelian sets $A\subset\mathbb{R}^{n}$ and $B\subset\mathbb{R}^{m}$
\[
\mathbb{P}\Bigl((X\in A)\cap(Y\in B)\Bigr)=\int_{A}\mathbb{P}(Y\in B\mid
X=x)\,\mathrm{d}\mu_{X}(x).
\]
Notice that the function $x\mapsto\mathbb{P}(Y\in B\mid X=x)$ is a
Radon-Nicodym derivative with respect to $\mu_{X}$ and is thus just defined
modulo the measure $\mu_{X}$. The measure $B\mapsto\mathbb{P}(Y\in B\mid X=x)$
is called the \textbf{conditional distribution of $Y$ given $X=x$}.}$^{)}$ of
$\mathcal{Y}(t_{k+1})$ given $({\mathcal{Y}}(t_{1}),\cdots,{\mathcal{Y}}%
(t_{k}))$ is equal to the conditional distribution of $\mathcal{Y}(t_{k+1})$
given ${\mathcal{Y}}(t_{k})$; in other words, for almost all $(y_{1}%
,\cdots,y_{k})$ with respect to the distribution of the random vector
$({\mathcal{Y}}(t_{1}),\cdots,{\mathcal{Y}}(t_{k}))$ and for all Borelian set
$B\subset\mathbb{R}^{n}$
\[
{\mathbb{P}}({\mathcal{Y}}(t_{k+1})\in B\mid{\mathcal{Y}}(t_{1})=y_{1}%
,\cdots,{\mathcal{Y}}(t_{k})=y_{k})={\mathbb{P}}({\mathcal{Y}}(t_{k+1})\in
B\mid{\mathcal{Y}}(t_{k})=y_{k}).
\]
We refer to the book \cite{Dy}, chapter 3, for a description of such
processes.}

{From now on, we consider a $\mathbb{R}^{n}$-valued Markov process
$(\mathcal{Y}(t))_{t\geq0}$ defined on $(\Omega,\mathcal{F},\mathbb{P})$ and
we assume the following conditions:}

\begin{enumerate}
\item[(i)] $M\subset\mathbb{R}^{n}$

\item[(ii)] for any integer $\ell\geq0$
\begin{equation}
Y_{\ell}:={\mathcal{Y}}(\ell)\in M\qquad{\mathbb{P}}{\mathrm{-almost\ surely}%
}. \label{discreteversion}%
\end{equation}
It readily follows that the sequence $Y=(Y_{\ell})_{\ell\geq0}$ is a
$M$-valued Markov chain.

{\color{green} }

\item[(iii)] for any integer $\ell\geq0$, the conditional distribution of
$(\mathcal{Y}(t))_{t\geq\ell}$ given $Y_{\ell}$ is equal to the one of
$(\mathcal{Y}(t))_{t\geq0}$ given $Y_{0}$; in other words, for any Borel set
$B\subset(\mathbb{R}^{n})^{\otimes\lbrack0,+\infty\lbrack}$ and any
$\lambda\in M$, one gets
\[
\mathbb{P}((\mathcal{Y}(t))_{t\geq\ell}\in B\mid Y_{\ell}=\lambda
)=\mathbb{P}((\mathcal{Y}(t))_{t\geq0}\in B\mid Y_{0}=\lambda).
\]

\end{enumerate}

In the following, we will assume that the initial distribution of the Markov
process $({\mathcal{Y}}(t))_{t\geq0}$ has full support, i.e. ${\mathbb{P}%
}({\mathcal{Y}}(0)=\lambda)>0$ for any $\lambda\in M$.

\subsection{Elementary random paths}

\label{subsec-ele_path} Consider a ${\mathbb{Z}}$-lattice ${P\subset
\mathbb{R}^{n}}$ with rank {\color{blue} $n$}.\ An \emph{elementary Littelmann
path} is a piecewise continuous linear map $\pi:[0,1]\rightarrow
P_{{\mathbb{R}}}$ such that $\pi(0)=0$ and $\pi(1)\in P$.\ Two paths which
coincide up to reparametrization are considered as identical.

The set ${\mathcal{F}}$ of continuous functions from $[0,1]$ to ${\mathbb{R}%
^{n}}$ is equipped with the norm $\left\Vert \cdot\right\Vert _{\infty}$ of
uniform convergence : for any $\pi\in{\mathcal{F}}$, on has $\left\Vert
\pi\right\Vert _{\infty}:=\sup_{t\in\lbrack0,1]}\left\Vert \pi(t)\right\Vert $
where $\left\Vert \cdot\right\Vert $ denotes the euclidean norm on
${P\subset\mathbb{R}^{n}}$. Let $B$ be a \emph{finite set of elementary paths}
and fix a probability distribution $p=(p_{\pi})_{\pi\in B}$ on $B$ such that
$p_{\pi}>0$ for any $\pi\in B$. Let $X$ be a random variable with values in
$B$ defined on a probability space $(\Omega,{\mathcal{F}},{\mathbb{P}})$ and
with distribution $p$ (in other words ${\mathbb{P}}(X=\pi)=p_{\pi}\ {\text{for
any }}\pi\in B).$ {The variable $X$ admits a moment of order $1$ defined by%
\[
m:={\mathbb{E}}(X)=\sum_{\pi\in B}p_{\pi}\pi.
\]
}The concatenation $\pi_{1}\ast\pi_{2}$ of two elementary paths $\pi_{1}$ and
$\pi_{2}$ is defined by
\[
\pi_{1}\ast\pi_{2}(t)=\left\{
\begin{array}
[c]{lll}%
\pi_{1}(2t) & {\text{ for }} & t\in\lbrack0,\frac{1}{2}],\\
\pi_{1}(1)+\pi_{2}(2t-1) & {\text{ for }} & t\in\lbrack\frac{1}{2},1].
\end{array}
\right.
\]
In the sequel, ${\mathcal{C}}$ is a closed convex cone in ${P\subset
\mathbb{R}^{n}}$.

\bigskip

Let $B$ be a set of elementary paths and $(X_{\ell})_{\ell\geq1}$ a sequence
of i.i.d. random variables with same law as $X$ where $X$ is the random
variable with values in $B$ introduced just above. We define a random process
${\mathcal{W}}$ as follows: for any $\ell\in{\mathbb{Z}}_{>0}$ and
$t\in\lbrack\ell,\ell+1]$
\[
{\mathcal{W}}(t):=X_{1}(1)+X_{2}(1)+\cdots+X_{\ell-1}(1)+X_{\ell}(t-\ell).
\]
The sequence of random variables $W=(W_{\ell})_{\ell\in{\mathbb{Z}_{\geq0}}%
}:=({\mathcal{W}}(\ell))_{\ell\geq0}$ is a random walk with set of increments
$I:=\{\pi(1)\mid\pi\in B\}$.

\section{Littelmann paths}

\subsection{Background on representation theory of Lie algebras}

Let $\mathfrak{g}$ be a simple finite-dimensional Lie algebra over
$\mathbb{C}$ of rank $n$ and $\mathfrak{g}=\mathfrak{g_{+}}\oplus
\mathfrak{h}\oplus\mathfrak{g_{-}}$ a triangular decomposition.\ We shall
follow the notation and convention of \cite{Bour}.\ According to the
Cartan-Killing classification, $\mathfrak{g}$ is characterized (up to
isomorphism) by its root system $R$. This root system is determined by the
previous triangular decomposition and realized in the euclidean space
$\mathbb{R}^{n}$. We denote by $\Delta_{+}=\{\alpha_{i}\mid i\in I\}$ the set
of simple roots of $\mathfrak{g},$ by $R_{+}$ the (finite) set of positive
roots.\ We then have $n=\mathrm{card}(\Delta_{+})$ and $R=R_{+}\cup R_{-}$
with $R_{-}=-R_{+}$. The root lattice of $\mathfrak{g}$ is the integral
lattice $Q=\bigoplus_{i=1}^{n}\mathbb{Z\alpha}_{i}.$ Write $\omega
_{i},i=1,\ldots,n$ for the fundamental weights associated to $\mathfrak{g}$.
The weight lattice associated to $\mathfrak{g}$ is the integral lattice
$P=\bigoplus_{i=1}^{n}\mathbb{Z\omega}_{i}.$ It can be regarded as an integral
sublattice of $\mathfrak{h}_{\mathbb{R}}^{\ast}$ (the real form of the dual
$\mathfrak{h}^{\ast}$ of $\mathfrak{h}$). We have $\dim(P)=\dim(Q)=n$ and
$Q\subset P$.

The cone of dominant weights for $\mathfrak{g}$ is obtained by considering the
positive integral linear combinations of the fundamental weights, that is
$P_{+}=\bigoplus_{i=1}^{n}\mathbb{Z}_{\geq0}\mathbb{\omega}_{i}.$ The
corresponding open Weyl chamber is the cone $\mathring{\mathcal{C}}%
=\bigoplus_{i=1}^{n}\mathbb{R}_{>0}\mathbb{\omega}_{i}$. We also introduce its
closure $\mathcal{C}=\bigoplus_{i=1}^{n}\mathbb{R}_{\geq0}\mathbb{\omega}_{i}%
$. In type $A$, we shall use the weight lattice of $\mathfrak{gl}_{n}$ rather
than that of $\mathfrak{sl}_{n}$ for simplicity.\ We also introduce the Weyl
group ${\mathsf{W}}$ of $\mathfrak{g}$ which is the group generated by the
orthogonal reflections $s_{i}$ through the hyperplanes perpendicular to the
simple root $\alpha_{i},i=1,\ldots,n$. Each $w\in{\mathsf{W}}$ may be
decomposed as a product of the $s_{i},i=1,\ldots,n.\;$All the minimal length
decompositions of $w$ have the same length $l(w)$.\ The group ${\mathsf{W}}$
contains a unique element $w_{0}$ of maximal length $l(w_{0})$ equal to the
number of positive roots of $\mathfrak{g}$, {this $w_{0}$ is an involution and
if $s_{i_{1}}\cdots s_{i_{r}}$ is a minimal length decomposition of $w_{0}$,
we have}
\begin{equation}
R_{+}=\{\alpha_{i_{1}},s_{i_{1}}\cdots s_{i_{a}}(\alpha_{i_{a+1}})\text{ with
}a=1,\ldots,r-1\}. \label{R+w0}%
\end{equation}

\begin{example}
\label{example-C2}The root system of $\mathfrak{g}=\mathfrak{sp}_{4}$ has rank
$2$. In the standard basis $(e_{1},e_{2})$ of the euclidean space
$\mathbb{R}^{2}$, we have $\omega_{1}=(1,0)$ and $\omega_{2}=(1,1).$ So
$P=\mathbb{Z}^{2}$ and $\mathcal{C}=\{(x_{1},x_{2})\in\mathbb{R}^{2}\mid
x_{1}\geq x_{2}\geq0\}$. The simple roots are $\alpha_{1}=e_{1}-e_{2}$ and
$\alpha_{2}=2e_{2}$. We also have $R_{+}=\{\alpha_{1},\alpha_{2},\alpha
_{1}+\alpha_{2},2\alpha_{1}+\alpha_{2}\}$. The Weyl group ${\mathsf{W}}$ is
the octahedral group with $8$ elements.\ It acts on $\mathbb{R}^{2}$ by
permuting the coordinates of the vectors and flipping their sign. More
precisely, for any $\beta=(\beta_{1},\beta_{2})\in\mathbb{R}^{2}$, we have
$s_{1}(\beta)=(\beta_{2},\beta_{1})$ and $s_{2}(\beta)=(\beta_{1},-\beta_{2}%
)$. The longest element is $w_{0}=-id=s_{1}s_{2}s_{1}s_{2}$. On easily
verifies we indeed have
\[
R_{+}=\{\alpha_{1},s_{1}s_{2}s_{1}(\alpha_{2})=\alpha_{2},s_{1}s_{2}%
(\alpha_{1})=\alpha_{1}+\alpha_{2},s_{1}(\alpha_{2})=2\alpha_{1}+\alpha
_{2}\}.
\]

\end{example}

We now summarize some properties of the action of ${\mathsf{W}}$ on the weight
lattice $P$. For any weight $\beta$, the orbit ${\mathsf{W}}\cdot\beta$ of
$\beta$ under the action of ${\mathsf{W}}$ intersects $P_{+}$ in a unique
point. We define a partial order on $P$ by setting $\mu\leq\lambda$ if
$\lambda-\mu$ belongs to $Q_{+}=\bigoplus_{i=1}^{n}\mathbb{Z}_{\geq
0}\mathbb{\alpha}_{i}$.

Let $U(\mathfrak{g})$ be the enveloping algebra associated to $\mathfrak{g}.$
Each finite dimensional $\mathfrak{g}$ (or $U(\mathfrak{g})$)-module $M$
admits a decomposition in weight spaces $M=\bigoplus_{\mu\in P}M_{\mu}$ where
\[
M_{\mu}:=\{v\in M\mid h(v)=\mu(h)v\text{ for any }h\in\mathfrak{h}\ {\text{and
some}\ \mu(h)\in\mathbb{C}}\}.
\]
This means that the action of any $h\in\mathfrak{h}$ on the weight space
$M_{\mu}$ is diagonal with eigenvalue $\mu(h)$.\ In particular, $(M\oplus
M^{\prime})_{\mu}=M_{\mu}\oplus M_{\mu}^{\prime}$.\ The Weyl group
${\mathsf{W}}$ acts on the weights of $M$ and for any $\sigma\in{\mathsf{W}}$,
we have $\dim M_{\mu}=\dim M_{\sigma\cdot\mu}$. For any $\gamma\in P$, let
$e^{\gamma}$ be the generator of the group algebra ${\mathbb{C}}[P]$
associated to $\gamma$. By definition, we have $e^{\gamma}e^{\gamma^{\prime}%
}=e^{\gamma+\gamma^{\prime}}$ for any $\gamma,\gamma^{\prime}\in P$ and the
group ${\mathsf{W}}$ acts on ${\mathbb{C}}[P]$ as follows: $w(e^{\gamma
})=e^{w(\gamma)}$ for any $w\in{\mathsf{W}}$ and any $\gamma\in P$.

The character of $M$ is the Laurent polynomial in ${\mathbb{C}}[P]$
$\mathrm{char}(M)(x):=\sum_{\mu\in P}\dim(M_{\mu})e^{\mu}$ where $\dim(M_{\mu
})$ is the dimension of the weight space $M_{\mu}$.

The irreducible finite dimensional representations of $\mathfrak{g}$ are
labelled by the dominant weights. For each dominant weight $\lambda\in P_{+},$
let $V(\lambda)$ be the irreducible representation of $\mathfrak{g}$
associated to $\lambda$. The category $\mathcal{C}$ of finite dimensional
representations of $\mathfrak{g}$ over $\mathbb{C}$ is semisimple: each module
decomposes into irreducible components. The category $\mathcal{C}$ is
equivariant to the (semisimple) category of finite dimensional $U(\mathfrak{g}%
)$-modules (over $\mathbb{C}$). Roughly speaking, this means that the
representation theory of $\mathfrak{g}$ is essentially identical to the
representation theory of the \emph{associative} algebra $U(\mathfrak{g)}%
$.\ Any finite dimensional $U(\mathfrak{g})$-module $M$ decomposes as a direct
sum of irreducible $M=\bigoplus_{\lambda\in P_{+}}V(\lambda)^{\oplus
m_{M,\lambda}}$ where $m_{M,\lambda}$ is the multiplicity of $V(\lambda)$ in
$M$. Here we slightly abuse the notation and also denote by $V(\lambda)$ the
irreducible f.d. $U(\mathfrak{g})$-module associated to $\lambda.$

When $M=V(\lambda)$ is irreducible, we set $s_{\lambda}:=\mathrm{char}%
(M)=\sum_{\mu\in P}K_{\lambda,\mu}e^{\mu}$ with $\dim(M_{\mu})=K_{\lambda,\mu
}.$ Then $K_{\lambda,\mu}\neq0$ only if $\mu\leq\lambda$. Recall also that the
characters can be computed from the Weyl character formula but we do not need
this approach in the sequel.

Given $\kappa,\mu$ in $P_{+}$ and a nonnegative integer $\ell$, we define the
tensor multiplicities $f_{\lambda/\mu,\kappa}^{\ell}$ by
\begin{equation}
V(\mu)\otimes V(\kappa)^{\otimes\ell}\simeq\bigoplus_{\lambda\in P_{+}%
}V(\lambda)^{\oplus f_{\lambda/\mu,\kappa}^{\ell}}. \label{def_f}%
\end{equation}
For $\mu=0$, we set $f_{\lambda,\kappa}^{\ell}=f_{\lambda/0,\kappa}^{\ell}$.
When there is no risk of confusion, we write simply $f_{\lambda/\mu}^{\ell}$
(resp. $f_{\lambda}^{\ell}$) instead of $f_{\lambda/\mu,\kappa}^{\ell}$ (resp.
$f_{\lambda,\kappa}^{\ell}$). We also define the multiplicities $m_{\mu
,\kappa}^{\lambda}$ by
\begin{equation}
V(\mu)\otimes V(\kappa)\simeq\bigoplus_{\mu\leadsto\lambda}V(\lambda)^{\oplus
m_{\mu,\kappa}^{\lambda}} \label{step}%
\end{equation}
where the notation $\mu\leadsto\lambda$ means that $\lambda\in P_{+}$ and
$V(\lambda)$ appears as an irreducible component of $V(\mu)\otimes V(\kappa)$.
We have in particular $m_{\mu,\kappa}^{\lambda}=f_{\lambda/\mu,\kappa}^{1}$.\ 

\subsection{Littelmann path model\label{subsecLit}}

We now give a brief overview of the Littelmann path model.\ We refer to
\cite{Lit1}, \cite{Lit2}, \cite{Lit3} and \cite{Kashi} for examples and a
detailed exposition. Consider a Lie algebra $\mathfrak{g}$ and its root system
realized in the euclidean space $P_{\mathbb{R}}=\mathbb{R}^{n}$. We fix a
scalar product $\langle\cdot,\cdot\rangle$ on $P_{\mathbb{R}}$ invariant under
${\mathsf{W}}$.\ For any root $\alpha$, we set $\alpha^{\vee}=\frac{2\alpha
}{\langle\alpha,\alpha\rangle}$. We define the notion of elementary continuous
piecewise linear paths in $P_{{\mathbb{R}}}$ as we did in
\S \ \ref{subsec-ele_path}.\ Let ${\mathcal{L}}$ be the set of elementary
paths $\eta$ having \emph{only rational turning points} (i.e. whose inflexion
points have rational coordinates) and ending in $P$ i.e. such that $\eta(1)\in
P$. We then define the weight of the path $\eta$ by $\mathrm{wt}(\eta
)=\eta(1)$. Given any path $\eta\in\mathcal{L}$, we define its reverse path
$r(\eta)\in\mathcal{L}$ by
\[
r(\eta)(t)=\eta(1-t)-\eta(1).
\]
Observe the map $r$ is an involution on $\mathcal{L}$.\ Littelmann {associated
to each simple root $\alpha_{i},i=1,\ldots,n,$ some root operators $\tilde
{e}_{i}$ and $\tilde{f}_{i}$ acting on ${\mathcal{L}}\cup\{{\mathbf{0}}\}$%
}.\ We do not need their complete definition in the sequel and refer to the
above mentioned papers for a complete review. Recall nevertheless that roots
operators $\tilde{e}_{i}$ and $\tilde{f}_{i}$ essentially act on a path $\eta$
by applying the symmetry $s_{\alpha}$ on parts of $\eta$ and we have%
\begin{equation}
\tilde{f}_{i}(\eta)=r\tilde{e}_{i}r(\eta).\label{fer}%
\end{equation}
These operators therefore preserve the length of t{he paths since the elements
of }${\mathsf{W}}$ are isometries.\ Also if $\tilde{f}_{i}(\eta)=\eta^{\prime
}\neq{\mathbf{0}}$, we have
\begin{equation}
\tilde{e}_{i}(\eta^{\prime})=\eta\text{ and }\mathrm{wt}(\tilde{f}_{i}%
(\eta))=\mathrm{wt}(\eta)-\alpha_{i}.\label{etil}%
\end{equation}
By drawing an arrow $\eta\overset{i}{\rightarrow}\eta^{\prime}$ between the
two paths $\eta,\eta^{\prime}$ of ${\mathcal{L}}$ as soon as $\tilde{f}%
_{i}(\eta)=\eta^{\prime}$ (or equivalently $\eta=\tilde{e}_{i}(\eta^{\prime}%
)$), we obtain a Kashiwara crystal graph with set of vertices ${\mathcal{L}}%
$.\ By abuse of notation, we yet denote it by ${\mathcal{L}}$ which so becomes
a colored oriented graph.\ For any $\eta\in{\mathcal{L}}$, we denote by
$B(\eta)$ the connected component of $\eta$ i.e. the subgraph of
${\mathcal{L}}$ generated by $\eta$ by applying operators $\tilde{e}_{i}$ and
$\tilde{f}_{i}$, $i=1,\ldots,n$.\ For any path $\eta\in{\mathcal{L}}$ and
$i=1,\ldots,n$, set $\varepsilon_{i}(\eta)=\max\{k\in{\mathbb{Z}}_{\geq0}%
\mid\tilde{e}_{i}^{k}(\eta)={\mathbf{0}}\}$ and $\varphi_{i}(\eta)=\max
\{k\in{\mathbb{Z}}_{\geq0}\mid\tilde{f}_{i}^{k}(\eta)={\mathbf{0}}\}$.

\bigskip

The set ${\mathcal{L}}_{\min{\mathbb{Z}}}$ of \textit{integral paths is the
set of }paths $\eta$ such that $m_{\eta}(i)=\min_{t\in\lbrack0,1]}%
\{\langle\eta(t),\alpha_{i}^{\vee}\rangle\}$ belongs to ${\mathbb{Z}}$ for any
$i=1,\ldots,n$. We also recall that we have
\[
\displaystyle{\mathcal{C}}=\{x\in{\mathfrak{h}}_{{\mathbb{R}}}^{\ast}%
\mid\langle x,\alpha_{i}^{\vee}\rangle\geq0\}\text{ and }\mathring{\mathcal
{C}}=\{x\in{\mathfrak{h}}_{{\mathbb{R}}}^{\ast}\mid\langle x,\alpha_{i}^{\vee
}\rangle>0\}{\text{.}}%
\]
Any path $\eta$ such that $\operatorname{Im}\eta\subset{\mathcal{C}}$ verifies
$m_{\eta}(i)=0$ so belongs to ${\mathcal{L}}_{\min{\mathbb{Z}}}$.\ One gets the

\begin{proposition}
\ \label{Prop_HP} Let $\eta$ and $\pi$ two paths in ${\mathcal{L}}%
_{\min{\mathbb{Z}}}$. Then

\begin{enumerate}
\item[(i)] the concatenation $\pi\ast\eta$ belongs to ${\mathcal{L}}%
_{\min{\mathbb{Z}}}$,

\item[(ii)] for any $i=1,\ldots,n$ we have
\begin{equation}
\tilde{e}_{i}(\eta\ast\pi)=\left\{
\begin{array}
[c]{ll}%
\eta\ast\tilde{e}_{i}(\pi) & {\text{if }}\varepsilon_{i}(\pi)>\varphi_{i}%
(\eta)\\
\tilde{e}_{i}(\eta)\ast\pi & {\text{otherwise,}}%
\end{array}
\right.  {\text{and }}\tilde{f}_{i}(\eta\ast\pi)=\left\{
\begin{array}
[c]{ll}%
\tilde{f}_{i}(\eta)\ast\pi & {\text{if }}\varphi_{i}(\eta)>\varepsilon_{i}%
(\pi)\\
\eta\ast\tilde{f}_{i}(\pi) & {\text{otherwise.}}%
\end{array}
\right.  \label{Ten_Prod}%
\end{equation}
In particular, $\tilde{e}_{i}(\eta\ast\pi)={\mathbf{0}}$ if and only if
$\tilde{e}_{i}(\eta)={\mathbf{0}}$ and $\varepsilon_{i}(\pi)\leq\varphi
_{i}(\eta)$ for any $i=1,\ldots,n$.

\item[(iii)] $\tilde{e}_{i}(\eta)={\mathbf{0}}$ for any $i=1,\ldots,n$ if and
only if $\operatorname{Im}\eta$ is contained in ${\mathcal{C}}$.
\end{enumerate}
\end{proposition}

The following theorem summarizes crucial results by Littelmann (see
\cite{Lit1}, \cite{Lit2} and \cite{Lit3}).\ 

\begin{theorem}
\label{Th_Littel}Consider $\lambda,\mu$ and $\kappa$ dominant weights and
choose arbitrarily elementary paths $\eta_{\lambda},\eta_{\mu}$ and
$\eta_{\kappa}$ in ${\mathcal{L}}$ such that $\operatorname{Im}\eta_{\lambda
}\subset{\mathcal{C}}$, $\operatorname{Im}\eta_{\mu}\subset{\mathcal{C}}$ and
$\operatorname{Im}\eta_{\kappa}\subset{\mathcal{C}}$ and joining respectively
$0$ to $\lambda$, $0$ to $\mu$ and $0$ to $\kappa$.

\begin{enumerate}
\item[(i)] We have $B(\eta_{\lambda}):=\{\tilde{f}_{i_{1}}\cdots\tilde
{f}_{i_{k}}\eta_{\lambda}\mid k\in${$\mathbb{Z}_{\geq0}$}${\text{ and }}1\leq
i_{1},\cdots,i_{k}\leq n\}\setminus\{{\mathbf{0}}\}.$

In particular ${\mathrm{wt}}(\eta)-{\mathrm{wt}}(\eta_{\lambda})\in Q_{+}$ for
any $\eta\in B(\eta_{\lambda})$.

\item[(ii)] All the paths in $B(\eta_{\lambda})$ have the same length than
$\eta_{\lambda}$.

\item[(iii)] The paths on $B(\eta_{\lambda})$ belong to ${\mathcal{L}}%
_{\min{\mathbb{Z}}}.$

\item[(iv)] If $\eta_{\lambda}^{\prime}$ is another elementary path from $0$
to $\lambda$ such that $\operatorname{Im}\eta_{\lambda}^{\prime}$ is contained
in ${\mathcal{C}}$, then $B(\eta_{\lambda})$ and $B(\eta_{\lambda}^{\prime})$
are isomorphic as oriented graphs i.e. there exists a bijection $\theta
:B(\eta_{\lambda})\rightarrow B(\eta_{\lambda}^{\prime})$ which commutes with
the action of the operators $\tilde{e}_{i}$ and $\tilde{f}_{i}$,
$i=1,\ldots,n$.

\item[(v)] We have
\begin{equation}
s_{\lambda}=\sum_{\eta\in B(\eta_{\lambda})}e^{\eta(1)}. \label{slambda}%
\end{equation}

\item[(vi)] For any $b\in B(\eta_{\lambda})$ we have ${\mathrm{wt}}%
(b)=\sum_{i=1}^{n}(\varphi_{i}(b)-\varepsilon_{i}(b))\omega_{i}.$

\item[(vii)] For any $i=1,\ldots,n$ and any $b\in B(\eta_{\lambda})$, let
$s_{i}(b)$ be the unique path in $B(\eta_{\lambda})$ such that
\[
\varphi_{i}(s_{i}(b))=\varepsilon_{i}(b){\text{ and }}\varepsilon_{i}%
(s_{i}(b))=\varphi_{i}(b)
\]
(in other words, $s_{i}$ acts on each $i$-chain ${\mathcal{C}}_{i}$ as the
symmetry with respect to the center of ${\mathcal{C}}_{i}$). The actions of
the $s_{i}$'s extend to an action\footnote{This action, defined from the
crystal structure on paths, should not be confused with the pointwise action
of the Weyl group on the paths.} of ${\mathsf{W}}$ on ${\mathcal{L}}$ which
stabilizes $B(\eta_{\lambda})$. In particular, for any $w\in{\mathsf{W}}$ and
any $b\in B(\eta_{\lambda})$, we have $w(b)\in B(\eta_{\lambda})$ and
${\mathrm{wt}}(w(b))=w({\mathrm{wt}}(b))$.

\item[(viii)] Given any integer $\ell\geq0$, set
\begin{equation}
B(\eta_{\mu})\ast B(\eta_{\kappa})^{\ast\ell}=\{\pi=\eta\ast\eta_{1}\ast
\cdots\ast\eta_{\ell}\in{\mathcal{L}}\mid\eta\in B(\eta_{\mu}){\text{ and }%
}\eta_{k}\in B(\eta_{\kappa})\ {\text{ for any }}k=1,\ldots,\ell\}.
\label{def_crysta_*}%
\end{equation}
The graph $B(\eta_{\mu})\ast B(\eta_{\kappa})^{\ast\ell}$ is contained in
${\mathcal{L}}_{\min{\mathbb{Z}}}.$

\item[(ix)] The multiplicity $m_{\mu,\kappa}^{\lambda}$ defined in
(\ref{step}) is equal to the number of paths of the form $\mu\ast\eta$ with
$\eta\in B(\eta_{\kappa})$ contained in ${\mathcal{C}}$.

\item[(x)] The multiplicity $f_{\lambda/\mu}^{\ell}$ defined in (\ref{def_f})
is equal to cardinality of the set
\[
H_{\lambda/\mu}^{\ell}:=\{\pi\in B(\eta_{\mu})\ast B(\eta_{\kappa})^{\ast\ell
}\mid\tilde{e}_{i}(\pi)=0\ {\text{ for any }}i=1,\ldots,n{\text{ and }}%
\pi(1)=\lambda\}.
\]
Each path $\pi=\eta\ast\eta_{1}\ast\cdots\ast\eta_{\ell}\in H_{\lambda/\mu
}^{\ell}$ verifies $\operatorname{Im}\pi\subset{\mathcal{C}}$ and $\eta
=\eta_{\mu}$.
\end{enumerate}
\end{theorem}

\begin{remarks}
\begin{enumerate}

\item[(i)] Combining (\ref{etil}) with assertions {(i) and (v)} of Theorem
\ref{Th_Littel}, one may check that the function $e^{-\lambda}s_{\lambda}$ is
in fact a polynomial in the variables $T_{i}=e^{-\alpha_{i}}$, namely
\begin{equation}
s_{\lambda}=e^{\lambda}S_{\lambda}(T_{1},\ldots,T_{n}) \label{defS}%
\end{equation}
where $S_{\lambda}\in{\mathbb{C}}[X_{1},\ldots,X_{n}]$.

\item[(ii)] Using {assertion (i) } of Theorem \ref{Th_Littel}, we obtain
$m_{\mu,\kappa}^{\lambda}\neq0$ only if $\mu+\kappa-\lambda\in Q_{+}$.
Similarly, when $f_{\lambda/\mu}^{\kappa,\ell}\neq0$ one necessarily has
$\mu+\ell\kappa-\lambda\in Q_{+}$.
\end{enumerate}
\end{remarks}

\section{Random paths from Littelmann paths}

\label{Sec_WP}In this Section we recall some results of \cite{LLP3}.\ We also
introduce the notion of central probability distribution on elementary
Littelmann paths and show these distributions coincide with those used in the
seminal works \cite{BBO}, \cite{OC1} and also in our previous papers
\cite{LLP}, \cite{LLP2},\cite{LLP3}.

\subsection{Central probability measure on trajectories}

Consider $\kappa\in P_{+}$ and a path $\pi_{\kappa}\in{\mathcal{L}}$ from $0$
to $\kappa$ such that $\operatorname{Im}\pi_{\kappa}$ is contained in
${\mathcal{C}}$. Let $B(\pi_{\kappa})$ be the connected component of
${\mathcal{L}}$ containing $\pi_{\kappa}$.\ Assume that $\{\pi_{1},\ldots
,\pi_{\ell}\}$ is a family of elementary paths in $B(\pi_{\kappa})$; the path
$\pi_{1}\otimes\cdots\otimes\pi_{\ell}$ of length $\ell$ is defined by: for
all $k\in\{1,\ldots,\ell-1\}$ and $t\in\lbrack k,k+1]$
\begin{equation}
\pi_{1}\otimes\cdots\otimes\pi_{\ell}(t)=\pi_{1}(1)+\cdots+\pi_{k}%
(1)+\pi_{k+1}(t-k). \label{tens_path}%
\end{equation}
Let $B(\pi_{\kappa})^{\otimes\ell}$ be the set of paths of the form $b=\pi
_{1}\otimes\cdots\otimes\pi_{\ell}$ where $\pi_{1},\ldots,\pi_{\ell}$ are
elementary paths in $B(\pi_{\kappa})$; there exists a bijection $\Delta$
between $B(\pi_{\kappa})^{\otimes\ell}$ and the set $B^{\ast\ell}(\pi_{\kappa
})$ of paths in ${\mathcal{L}}$ obtained by concatenations of $\ell$ paths of
$B(\pi_{\kappa})$:
\begin{equation}
\Delta:\left\{
\begin{array}
[c]{clc}%
B(\pi_{\kappa})^{\otimes\ell} & \longrightarrow & B(\pi_{\kappa})^{\ast\ell}\\
\pi_{1}\otimes\cdots\otimes\pi_{\ell} & \longmapsto & \pi_{1}\ast\cdots\ast
\pi_{\ell}%
\end{array}
\right.  . \label{DefDelta}%
\end{equation}
In fact $\pi_{1}\otimes\cdots\otimes\pi_{\ell}$ and $\pi_{1}\ast\cdots\ast
\pi_{\ell}$ coincide up to a reparametrization and we define the weight of
$b=\pi_{1}\otimes\cdots\otimes\pi_{\ell}$ setting
\[
{\mathrm{wt}}(b):={\mathrm{wt}}(\pi_{1})+\cdots+{\mathrm{wt}}(\pi_{\ell}%
)=\pi_{1}(1)+\cdots+\pi_{\ell}(1).
\]
The involution $r$ on $\eta\in B(\pi_{\kappa})^{\otimes\ell}$ is such that
\[
r(\eta)(t)=\eta(\ell-t)-\eta(0)
\]
for any $t\in\lbrack0,\ell]$.

Consider $p$ a probability distribution on $B(\pi_{\kappa})$ such that
$p_{\pi}>0$ for any $\pi\in B(\pi_{\kappa})$. For any integer $\ell\geq1$, we
endow $B(\pi_{\kappa})^{\otimes\ell}$ with the product density $p^{\otimes
\ell}$. That is we set $p_{\pi}^{\otimes\ell}=p_{\pi_{1}}\times\cdots\times
p_{\pi_{\ell}}$ for any $\pi=\pi_{1}\otimes\cdots\otimes\pi_{\ell}\in
B(\pi_{\kappa})^{\otimes\ell}$.\ Here, we follow the classical construction of
a Bernoulli process.\ Write $\Pi_{\ell}:B(\pi_{\kappa})^{\otimes\ell
}\rightarrow B(\pi_{\kappa})^{\otimes\ell-1}$ the projection defined by
$\Pi_{\ell}(\pi_{1}\otimes\cdots\otimes\pi_{\ell-1}\otimes\pi_{\ell})=\pi
_{1}\otimes\cdots\otimes\pi_{\ell-1}$; the sequence $(B(\pi_{\kappa}%
)^{\otimes\ell},\Pi_{\ell},p^{\otimes\ell})_{\ell\geq1}$ is a projective
system of probability spaces. We denote by $\Omega=(B(\pi_{\kappa}%
)^{\otimes\mathbb{Z}_{\geq0}},p^{\otimes{\mathbb{Z}_{\geq0}}})$ its projective
limit. The elements of $B(\pi_{\kappa})^{\otimes\mathbb{Z}_{\geq0}}$ are
infinite sequences $\omega=(\pi_{\ell})_{\ell\geq1}$ we call trajectories. By
a slight abuse of notation, we will write $\Pi_{\ell}(\omega)=\pi_{1}%
\otimes\cdots\otimes\pi_{\ell}$. We also write $\mathbb{P}=p^{\otimes
{\mathbb{Z}_{\geq0}}}$ for short. For any $b\in B(\pi_{\kappa})^{\otimes\ell}%
$, we denote by $U_{b}=\{\omega\in\Omega\mid\Pi_{\ell}(\omega)=b\}$ the
cylinder defined by $\pi$ in $\Omega$.

\begin{definition}
\label{Def_Central}The probability distribution $\mathbb{P}=p^{\otimes
{\mathbb{Z}_{\geq0}}}$ is \emph{central} on $\Omega$ when for any $\ell\geq1$
and any vertices $b$ and $b^{\prime}$ in $B(\pi_{\kappa})^{\otimes\ell}$ such
that $\mathrm{wt}(b)=\mathrm{wt}(b^{\prime})$ we have $\mathbb{P}%
(U_{b})=\mathbb{P}(U_{b^{\prime}})$.
\end{definition}

\begin{remark}
The probability distribution $\mathbb{P}$ is \emph{central} when for any
integer $\ell\geq1$ and any vertices $b,b^{\prime}$ in $B(\pi_{\kappa
})^{\otimes\ell}$ such that $\mathrm{wt}(b)=\mathrm{wt}(b^{\prime})$, we have
$p_{b}^{\otimes\ell}=p_{b^{\prime}}^{\otimes\ell}$. We indeed have
$U_{b}=b\otimes\Omega$ and $U_{b^{\prime}}=b\otimes\Omega.$ Hence
$\mathbb{P}(U_{b})=p_{b}^{\otimes\ell}$ and $\mathbb{P}(U_{b^{\prime}%
})=p_{b^{\prime}}^{\otimes\ell}$.
\end{remark}

\bigskip

The following proposition shows that $\mathbb{P}$ can only be central when the
probability distribution $p$ on $B(\pi_{\kappa})$ is compatible with the
graduation of $B(\pi_{\kappa})$ by the set of simple roots. This justifies the
restriction we did in \cite{LLP} and \cite{LLP3} on the probability
distributions we have considered on $B(\pi_{\kappa})$. This restriction will
also be relevant in the remaining of this paper.

\begin{proposition}
The following assertions are equivalent

\begin{enumerate}
\item[(i)] The probability distribution $\mathbb{P}$ is central.

\item[(ii)] There exists an $n$-tuple $\tau=(\tau_{1},\ldots,\tau_{n}%
)\in]0,+\infty\lbrack^{n}$ such that for each arrow $\pi\overset
{i}{\rightarrow}\pi^{\prime}$ in $B(\pi_{\kappa})$, we have the relation
$p_{\pi^{\prime}}=p_{\pi}\times\tau_{i}.$
\end{enumerate}
\end{proposition}

\begin{proof}
Assume probability distribution $\mathbb{P}$ is central. For any path $\pi\in
B(\pi_{\kappa})$, we define the depth $d(\pi)$ as the number of simple roots
appearing in the decomposition of $\kappa-\mathrm{wt}(\pi)$ on the basis of
simple roots (see {assertion (i)} of Theorem \ref{Th_Littel}). This is also
the length of any path joining $\pi_{\kappa}$ to $\pi$ in the crystal graph
$B(\pi_{\kappa})$. We have to prove that $\frac{p_{\pi^{\prime}}}{p_{\pi}}$ is
a constant depending only on $i$ as soon as we have an arrow $\pi\overset
{i}{\rightarrow}\pi^{\prime}$ in $B(\pi_{\kappa})$. For any $k\geq1$, we set
$B(\pi_{\kappa})_{k}=\{\pi\in B(\pi_{\kappa})\mid d(\pi)\leq k\}$. We will
proceed by induction and prove that $\frac{p_{\pi^{\prime}}}{p_{\pi}}$ is a
constant depending only on $i$ as soon as there is an arrow $\pi\overset
{i}{\rightarrow}\pi^{\prime}$ in $B(\pi_{\kappa})_{k}$. This is clearly true
in $B(\pi_{\kappa})_{1}$ since there is at most one arrow $i$ starting from
$\pi_{\kappa}$. Assume, the property is true in $B(\pi_{\kappa})_{k}$ with
$k\geq1$. Consider $\pi^{\prime}$ in $B(\pi_{\kappa})_{k+1}$ and an arrow
$\pi\overset{i}{\rightarrow}\pi^{\prime}$ in $B(\pi_{\kappa})_{k+1}$. We must
have $\pi\in B(\pi_{\kappa})_{k}$. If $B(\pi_{\kappa})_{k}$ does not contains
any arrow $\overset{i}{\rightarrow}$, there is nothing to verify. So assume
there is at least an arrow $\pi_{1}\overset{i}{\rightarrow}\pi_{2}$ in
$B(\pi_{\kappa})_{k}$. In $B(\pi_{\kappa})^{\otimes2}$, we have $\mathrm{wt}%
(\pi_{1}\otimes\pi^{\prime})=\mathrm{wt}(\pi_{1})+\mathrm{wt}(\pi)-\alpha_{i}$
since $\mathrm{wt}(\pi^{\prime})=\mathrm{wt}(\pi)-\alpha_{i}$. Similarly, we
have $\mathrm{wt}(\pi_{2}\otimes\pi)=\mathrm{wt}(\pi_{1})-\alpha
_{i}+\mathrm{wt}(\pi)$ since $\mathrm{wt}(\pi_{2})=\mathrm{wt}(\pi_{1}%
)-\alpha_{i}$. Thus $\mathrm{wt}(\pi_{1}\otimes\pi^{\prime})=\mathrm{wt}%
(\pi_{2}\otimes\pi)$. Since $\mathbb{P}$ is central, we deduce from the above
remark the equality $p^{\otimes2}(\pi_{1}\otimes\pi^{\prime})=p^{\otimes2}%
(\pi_{2}\otimes\pi)$. This yields $p_{\pi_{1}}p_{\pi^{\prime}}=p_{\pi_{2}%
}p_{\pi}$. Hence $\frac{p_{\pi^{\prime}}}{p_{\pi}}=\frac{p_{\pi_{2}}}%
{p_{\pi_{1}}}$. So by our induction hypothesis, $\frac{p_{\pi^{\prime}}%
}{p_{\pi}}$ is equal to a constant which only depends on $i$.

Conversely, assume there exists an $n$-tuple $\tau=(\tau_{1},\ldots,\tau
_{n})\in]0,+\infty\lbrack^{n}$ such that for each arrow $\pi\overset
{i}{\rightarrow}\pi^{\prime}$ in $B(\pi_{\kappa})$, we have the relation
$p_{\pi^{\prime}}=p_{\pi}\times\tau_{i}.$ Consider vertices $b,b^{\prime}$ in
$B(\pi_{\kappa})^{\otimes\ell}$ such that $\mathrm{wt}(b)=\mathrm{wt}%
(b^{\prime})$. Since $b$ and $b^{\prime}$ have the same weight, we derive from
(\ref{etil}) that the paths from $\pi_{\kappa}$ to $b$ and the paths from
$\pi_{\kappa}$ to $b^{\prime}$ contain the same number (says $a_{i}$) of
arrows $\overset{i}{\rightarrow}$ for any $i=1,\ldots,n$. We therefore have
$p_{b}=p_{b^{\prime}}=p_{\pi_{\kappa}}\tau_{1}^{a_{1}}\cdots\tau_{n}^{a_{n}}$
and the probability distribution $\mathbb{P}$ is central.
\end{proof}

\subsection{Central probability distributions on elementary paths}

\label{Probability distribution on elementary paths}

In the remaining of the paper, we fix the $n$-tuple $\tau=(\tau_{1}%
,\ldots,\tau_{n})\in]0,+\infty\lbrack^{n}$ and assume that $\mathbb{P}$ is a
central distribution on $\Omega$ defined from $\tau$ (in the sense of
Definition \ref{Def_Central}.\ For any $u=u_{1}\alpha_{1}+\cdots+u_{n}%
\alpha_{n}\in Q$, we set $\tau^{u}=\tau_{1}^{u_{1}}\cdots\tau_{n}^{u_{n}}$.
Since the root and weight lattices have both rank $n$, any weight $\beta\in P$
also decomposes on the form $\beta=\beta_{1}\alpha_{1}+\cdots+\beta_{n}%
\alpha_{n}$ with possibly non integral coordinates $\beta_{i}$. The transition
matrix between the bases $\{\omega_{i},i=1,\ldots,n\}$ and $\{\alpha
_{i},i=1,\ldots,n\}$ (regarded as bases of $P_{\mathbb{R}}$) being the Cartan
matrix of $\mathfrak{g}$ whose entries are integers, the coordinates
$\beta_{i}$ are rational. We will also set $\tau^{\beta}=\tau_{1}^{\beta_{1}%
}\cdots\tau_{n}^{\beta_{n}}$.

Let $\pi\in B(\pi_{\kappa})$: by {assertion (i)} of Theorem \ref{Th_Littel},
one gets
\[
\pi(1)={\mathrm{wt}}(\pi)=\kappa-\sum_{i=1}^{n}u_{i}(\pi)\alpha_{i}%
\]
where $u_{i}(\pi)\in${$\mathbb{Z}_{\geq0}$} for any $i=1,\ldots,n$. We define
$S_{\kappa}(\tau):=S_{\kappa}(\tau_{1},\ldots,\tau_{n})=\sum_{\pi\in
B(\pi_{\kappa})}\tau^{\kappa-{\mathrm{wt}}(\pi)}.$

\begin{definition}
We define the probability distribution $p=(p_{\pi})_{\pi\in B(\pi_{\kappa})}$
on $B(\pi_{\kappa})$ associated to $\tau$ by setting $\displaystyle p_{\pi
}=\frac{\tau^{\kappa-{\mathrm{wt}}(\pi)}}{S_{\kappa}(\tau)}.$
\end{definition}

\begin{remark}
By {assertion (iii)} of Theorem \ref{Th_Littel}, for $\pi_{\kappa}^{\prime}$
another elementary path from $0$ to $\kappa$ such that $\operatorname{Im}%
\pi_{\kappa}^{\prime}$ is contained in ${\mathcal{C}}$, there exists an
isomorphism $\Theta$ between the crystals $B(\pi_{\kappa})$ and $B(\pi
_{\kappa}^{\prime})$.\ For $p^{\prime}$ the central probability distribution
defined from $\tau$ on $B(\pi_{\kappa}^{\prime})$, one gets $p_{\pi}%
=p_{\Theta(\pi)}^{\prime}$ for any $\pi\in B(\pi_{\kappa})$. Therefore, the
probability distributions we use on the graph $B(\pi_{\kappa})$ are invariant
by crystal isomorphisms and also the probabilistic results we will establish
in the paper.
\end{remark}

\bigskip

The following proposition gathers results of \cite{LLP} (Lemma 7.2.1) and
\cite{LLP3} (Proposition 5.4) . Recall that $m=\sum_{\pi\in B(\pi_{\kappa}%
)}p_{\pi}\pi$. We set $\overline{m}=m(1)$.

\begin{proposition}
\ \label{Prop_Length}

\begin{enumerate}
\item[(i)] We have $\overline{m}\in\mathring{\mathcal{C}}$ if and only if
$\tau_{i}\in]0,1[$ for any $i=1,\ldots,n$.

\item[(ii)] Denote by $L$ the common length of the paths in $B(\pi_{\kappa})$.
Then, the length of $m$ is less or equal to $L$.
\end{enumerate}
\end{proposition}

Set $\mathcal{M}_{\kappa}=\{\overline{m}\mid\tau=(\tau_{1},\ldots,\tau_{n}%
)\in]0,+\infty\lbrack\}$ be the set of all vectors $\overline{m}$ obtained
from the central distributions defined on $B(\pi_{\kappa})$. Observe that
$\mathcal{M}_{\kappa}$ only depends on $\kappa$ and not of the choice of the
highest path $\pi_{\kappa}$. This is the set of possible mean obtained from
central probability distributions defined on $B(\pi_{\kappa})$. We will also
need the set%
\begin{equation}
\mathcal{D}_{\kappa}=\mathcal{M}_{\kappa}\cap\mathring{\mathcal{C}%
}=\{\overline{m}\in\mathcal{M}_{\kappa}\mid\tau_{i}\in]0,1[,i=1,\ldots,n\}
\label{DefDkappa}%
\end{equation}
of drifts in $\mathring{\mathcal{C}}$.

\begin{example}
\label{exam_typ}We resume Example \ref{example-C2} and consider the Lie
algebra $\mathfrak{g=sp}_{4}$ of type $C_{2}$ for which $P=\mathbb{Z}^{2}$ and
$\mathcal{C}=\{(x_{1},x_{2})\in\mathbb{R}^{2}\mid x_{1}\geq x_{2}\geq0\}$.

For $\kappa=\omega_{1}$ and $\pi_{\kappa}$ the line between $0$ and
$\varepsilon_{1}$, we get $B(\pi_{\kappa})=\{\pi_{1},\pi_{2},\pi_{\overline
{2}},\pi_{\overline{1}}\}$ where each $\pi_{a}$ is the line between $0$ and
$\varepsilon_{a}$ (with the convention $\varepsilon_{\overline{2}%
}=-\varepsilon_{2}$ and $\varepsilon_{\overline{1}}=-\varepsilon_{1}$). The
underlying crystal graph is%
\[
\pi_{1}\overset{1}{\rightarrow}\pi_{2}\overset{2}{\rightarrow}\pi
_{\overline{2}}\overset{1}{\rightarrow}\pi_{\overline{1}}\text{.}%
\]
For $(\tau_{1},\tau_{2})\in]0,+\infty\lbrack^{2}$, we obtain the probability
distribution on $B(\pi_{\kappa})$%
\begin{align*}
p_{\pi_{1}}  &  =\frac{1}{1+\tau_{1}+\tau_{1}\tau_{2}+\tau_{1}^{2}\tau_{2}%
},p_{\pi_{2}}=\frac{\tau_{1}}{1+\tau_{1}+\tau_{1}\tau_{2}+\tau_{1}^{2}\tau
_{2}},\\
p_{\pi_{\overline{2}}}  &  =\frac{\tau_{1}\tau_{2}}{1+\tau_{1}+\tau_{1}%
\tau_{2}+\tau_{1}^{2}\tau_{2}}\text{ and }p_{\pi_{\overline{2}}}=\frac
{\tau_{1}^{2}\tau_{2}}{1+\tau_{1}+\tau_{1}\tau_{2}+\tau_{1}^{2}\tau_{2}}.
\end{align*}
So we have
\[
\overline{m}=\frac{1}{1+\tau_{1}+\tau_{1}\tau_{2}+\tau_{1}^{2}\tau_{2}%
}((1-\tau_{1}^{2}\tau_{2})\varepsilon_{1}+(\tau_{1}-\tau_{1}\tau
_{2})\varepsilon_{2}).
\]
When the pair $(\tau_{1},\tau_{2})$ runs over $]0,1[^{2}$, one verifies by a
direct computation that $\mathcal{D}_{\kappa}$ coincide with the interior of
the triangle with vertices $0,\varepsilon_{1},\varepsilon_{2}$.
\end{example}

\begin{remark}
In the previous example, it is easy to show by a direct calculation that the
adherence $\overline{\mathcal{M}}_{\kappa}$ of $\mathcal{M}_{\kappa}$ is the
convex hull of the weight $\{\pm\varepsilon_{1},\pm\varepsilon_{2}\}$ of the
representation $V(\omega_{1})$ considered (i.e. the interior of the square
with vertices $\{\pm\varepsilon_{1},\pm\varepsilon_{2}\}$).\ In general, one
can show that $\overline{\mathcal{M}}_{\kappa}$ is contained in the convex
hull of the weights of $V(\kappa)$. The problem of determining, for any
dominant weight $\kappa$, wether or not both sets coincide seems to us
interesting and not immediate.
\end{remark}

\subsection{Random paths of arbitrary length}

\label{Subsec-RPW}With the previous convention, the product probability
measure $p^{\otimes\ell}$ on $B(\pi_{\kappa})^{\otimes\ell}$ satisfies
\begin{equation}
p^{\otimes\ell}(\pi_{1}\otimes\cdots\otimes\pi_{\ell})=p(\pi_{1})\cdots
p(\pi_{\ell})=\frac{\tau^{\ell\kappa-(\pi_{1}(1)+\cdots+\pi_{\ell}(1))}%
}{S_{\kappa}(\tau)^{\ell}}=\frac{\tau^{\ell\kappa-{\mathrm{wt}}(b)}}%
{S_{\kappa}(\tau)^{\ell}}. \label{potimesell}%
\end{equation}
Let $(X_{\ell})_{\ell\geq1}$ be a sequence of i.i.d. random variables {with
values in $B(\pi_{\kappa})$ and law $p=(p_{\pi})_{\pi\in B(\pi_{\kappa})}$;
for any $\ell\geq1$ we thus gets }
\begin{equation}
\mathbb{P}(X_{\ell}=\pi)=p_{\pi}\text{ for any }\pi\in B(\pi_{\kappa}).
\label{defX}%
\end{equation}
Consider $\mu\in P$. The random path ${\mathcal{W}}$ starting at $\mu$ is
defined from the probability space $\Omega$ with values in $P_{\mathbb{R}}$
by
\[
{\mathcal{W}}(t):=\Pi_{\ell}({\mathcal{W}})(t)=\mu+(X_{1}\otimes\cdots\otimes
X_{\ell-1}\otimes X_{\ell})(t){\text{ for }}t\in\lbrack\ell-1,\ell].
\]
For any integer $\ell\geq1$, we set $W_{\ell}=\mathcal{W}(\ell)$.\ The
sequence $W=(W_{\ell})_{\ell\geq1}$ defines a random walk starting at
$W_{0}=\mu$ whose increments are the weights of the representation $V(\kappa
)$. The following proposition was established in \cite{LLP3} (see Proposition 4.6).

\begin{proposition}
\ \label{Prop_util}

\begin{enumerate}
\item[(i)] For any $\beta,\eta\in P$, one gets
\[
{\mathbb{P}}(W_{\ell+1}=\beta\mid W_{\ell}=\eta)=K_{\kappa,\beta-\eta}%
\frac{\tau^{\kappa+\eta-\beta}}{S_{\kappa}(\tau)}.
\]

\item[(ii)] Consider $\lambda,\mu\in P^{+}$ we have
\[
{\mathbb{P}}(W_{\ell}=\lambda,W_{0}=\mu,{\mathcal{W}}(t)\in{\mathcal{C}%
}\ {\text{ for any }}t\in\lbrack0,\ell])=f_{\lambda/\mu}^{\ell}\frac
{\tau^{\ell\kappa+\mu-\lambda}}{S_{\kappa}(\tau)^{\ell}}.
\]
In particular
\[
{\mathbb{P}}(W_{\ell+1}=\lambda,W_{\ell}=\mu,{\mathcal{W}}(t)\in{\mathcal{C}%
}\ {\text{ for any }}t\in\lbrack\ell,\ell+1])=m_{\mu,\kappa}^{\lambda}%
\frac{\tau^{\kappa+\mu-\lambda}}{S_{\kappa}(\tau)}.
\]

\end{enumerate}
\end{proposition}

\subsection{The generalized Pitman transform}

\label{subsec_Pitman}By {assertion (viii)} of Theorem \ref{Th_Littel}, we know
that $B(\pi_{\kappa})^{\otimes\ell}$ is contained in ${\mathcal{L}}%
_{\min{\mathbb{Z}}}$.\ Therefore, if we consider a path $\eta\in B(\pi
_{\kappa})^{\otimes\ell}$, its connected component $B(\eta)$ is contained in
${\mathcal{L}}_{\min{\mathbb{Z}}}$.\ Now, if $\eta^{h}\in B(b)$ is such that
$\tilde{e}_{i}(\eta^{h})=0$ for any $i=1,\ldots,n$, we should have
$\operatorname{Im}\eta^{h}\subset{\mathcal{C}}$ by assertion (iii) of
Proposition \ref{Prop_HP}. {Assertion (iii)} of Theorem \ref{Th_Littel} thus
implies that $\eta^{h}$ is the unique highest weight path in $B(\eta
)=B(\eta^{h})$.\ Similarly, there is a unique lowest path $\eta_{l}$ in
$B(\eta)$ such that $\tilde{f}_{i}(\eta_{l})=0$ for any $i=1,\ldots,n$.\ This
permits to define the \emph{generalized Pitman} transform on $B(\pi_{\kappa
})^{\otimes\ell}$ as the map {$\mathcal{P}$} which associates to any $\eta\in
B(\pi_{\kappa})^{\otimes\ell}$ the unique path {$\mathcal{P}$}$(\eta)\in
B(\eta)$ such that $\tilde{e}_{i}({\mathcal{P}}(\eta))=0$ for any
$i=1,\ldots,n$.\ By definition, we have $\operatorname{Im}${$\mathcal{P}$%
}$(\eta)\subset{\mathcal{C}}$ and {$\mathcal{P}$}$(\eta)(\ell)\in P_{+}$. One
can also define a dual Pitman transform $\mathcal{E}$ which associates to any
$\eta\in B(\pi_{\kappa})^{\otimes\ell}$ the unique path {$\mathcal{E}$}%
$(\eta)\in B(\eta)$ such that $\tilde{f}_{i}({\mathcal{E}}(\eta))=0$ for any
$i=1,\ldots,n$.\ By (\ref{fer}), we have in fact
\[
\mathcal{E}=r\mathcal{P}r
\]
As observed in \cite{BBO} the path transformation {$\mathcal{P}$} can be made
more explicit (recall we have assumed that $\mathfrak{g}$ is
finite-dimensional). Consider a simple reflection $\alpha$. The Pitman
transformation {$\mathcal{P}$}$_{\alpha}:B(\pi_{\kappa})^{\otimes\ell
}\rightarrow B(\pi_{\kappa})^{\otimes\ell}$ associated to $\alpha$ is defined
by
\begin{equation}
{\mathcal{P}}_{\alpha}(\eta)(t)=\eta(t)-2\inf_{s\in\lbrack0,t]}\langle
\eta(s),\frac{\alpha}{\left\Vert \alpha\right\Vert ^{2}}\rangle\alpha
=\eta(t)-\inf_{s\in\lbrack0,t]}\langle\eta(s),\alpha^{\vee}\rangle
\alpha\label{defPalpha}%
\end{equation}
for any $\eta\in B(\pi_{\kappa})^{\otimes\ell}$ and any $t\in\lbrack0,\ell]$.
Also define the dual transform $\mathcal{E}_{\alpha}:=r\mathcal{P}_{\alpha}r$
on $B(\pi_{\kappa})^{\otimes\ell}$.\ One verifies easily that we have in fact%
\begin{equation}
{\mathcal{E}}_{\alpha}(\eta)(t)=\eta(t)-\inf_{s\in\lbrack t,\ell]}\langle
\eta(s),\alpha^{\vee}\rangle\alpha+\inf_{s\in\lbrack0,\ell]}\langle
\eta(s),\alpha^{\vee}\rangle\alpha. \label{def_Ealpha}%
\end{equation}
Let $w_{0}$ be the maximal length element of ${\mathsf{W}}$ and fix a
decomposition $w_{0}=s_{i_{1}}\cdots s_{i_{r}}$ of $w_{0}$ as a product a
reflections associated to simple roots.

\begin{proposition}
[\cite{BBO}]\label{Prop-BBO}For any path $\eta\in B(\pi_{\kappa})^{\otimes
\ell}$, we have
\begin{equation}
{\mathcal{P}}(\eta)={\mathcal{P}}_{\alpha_{i_{1}}}\cdots{\mathcal{P}}%
_{\alpha_{i_{r}}}(\eta)\text{ and }{\mathcal{E}}(\eta)={\mathcal{E}}%
_{\alpha_{i_{1}}}\cdots{\mathcal{E}}_{\alpha_{i_{r}}}(\eta).\label{def_P}%
\end{equation}
Moreover, $\mathcal{P}$ and $\mathcal{E}$ do not depend on the decomposition
of $w_{0}$ chosen.
\end{proposition}

\begin{remarks}
\ 

\begin{enumerate}
\item Since {$\mathcal{P}$}$(\eta)$ corresponds to the highest weight vertex
of the crystal $B(\eta)$, we have {$\mathcal{P}$}$^{2}(\eta)=${$\mathcal{P}$%
}$(\eta)$.

\item One easily verifies that each transformation {$\mathcal{P}$}$_{\alpha}$
is continuous for the topology of uniform convergence on the space of
continuous maps from $[0,\ell]$ to $\mathbb{R}$. Hence {$\mathcal{P}$} is also
continuous for this topology.

\item Assume $\eta\in B(\eta_{\lambda})\subset B(\pi_{\kappa})^{\otimes\ell}$
where $\eta_{\lambda}$ is the highest weight path of $B(\eta_{\lambda})$. Then
$\eta^{\lambda}=w_{0}(\eta_{\lambda})$ (the action of ${\mathsf{W}}$ is that
of Theorem \ref{Th_Littel}) is the lowest weight path in $B(\eta_{\lambda}%
)$.\ In this particular case, one can show that we have in fact
\begin{equation}
\mathcal{P}_{i_{a+1}}\cdots\mathcal{P}_{i_{r}}(\eta^{\lambda})=s_{i_{a+1}%
}\cdots s_{i_{r}}(\eta^{\lambda})\text{ and }\mathcal{E}_{i_{a+1}}%
\cdots\mathcal{E}_{i_{r}}(\eta_{\lambda})=s_{i_{a+1}}\cdots s_{i_{r}}%
(\eta_{\lambda})\label{PEHW}%
\end{equation}
for any $a=1,\ldots,r-1$.
\end{enumerate}
\end{remarks}

Let ${\mathcal{W}}$ be the random path of \S \ \ref{Subsec-RPW}.\ We define
the random process ${\mathcal{H}}$ setting
\begin{equation}
{\mathcal{H}}={\mathcal{P}}({\mathcal{W}}).\label{defH}%
\end{equation}
For any $\ell\geq1$, we set $H_{\ell}:={\mathcal{H}}(\ell).$ The following
Theorem was established in \cite{LLP3}.

\begin{theorem}
\label{Th_LawH}

\begin{enumerate}
\item[(i)] The random sequence $H:=(H_{\ell})_{\ell\geq1}$ is a Markov chain
with transition matrix
\begin{equation}
\Pi(\mu,\lambda)=\frac{S_{\lambda}(\tau)}{S_{\kappa}(\tau)S_{\mu}(\tau)}%
\tau^{\kappa+\mu-\lambda}m_{\mu,\kappa}^{\lambda} \label{MatrixH}%
\end{equation}
where $\lambda,\mu\in P_{+}$.

\item[(ii)] Assume $\eta\in B(\pi_{\kappa})^{\otimes\ell}$ is a highest weight
path of weight $\lambda$. Then
\[
\mathbb{P}(\mathcal{W}_{\ell}=\eta)=\frac{\tau^{\ell\kappa-\lambda}S_{\lambda
}(\tau)}{S_{\kappa}(\tau)^{\ell}}%
\]

\end{enumerate}
\end{theorem}

We shall also need the asymptotic behavior of the tensor product
multiplicities established in \cite{LLP3}.

\begin{theorem}
\label{Th_f_assymp}Assume $\overline{m}\in\mathcal{D}_{\kappa}$ (see
(\ref{DefDkappa})). For any $\mu\in P$ and any sequence of dominant weights of
the form $\lambda^{(\ell)}=\ell\overline{m}+o(\ell)$, we have

\begin{enumerate}
\item[(i)] $\underset{\ell\rightarrow+\infty}{\lim}\frac{f_{\lambda^{(\ell
)}-\gamma}^{\ell}}{f_{\lambda^{(\ell)}}^{\ell}}=\tau^{-\gamma}$ for any
$\gamma\in P$.

\item[(ii)] $\underset{\ell\rightarrow+\infty}{\lim}\frac{f_{\lambda^{(\ell
)}/\mu}^{\ell}}{f_{\lambda^{(\ell)}}^{\ell}}=\tau^{-\mu}S_{\mu}(\tau)$.
\end{enumerate}
\end{theorem}

\begin{corollary}
Under the assumptions of the previous theorem, we also have%
\[
\underset{\ell\rightarrow+\infty}{\lim}\frac{f_{\lambda^{(\ell)}}^{\ell
-\ell_{0}}}{f_{\lambda^{(\ell)}}^{\ell}}=\frac{1}{\tau^{-\ell_{0}\kappa
}S_{\kappa}^{\ell_{0}}(\tau)}%
\]
for any nonnegative integer $\ell_{0}$.
\end{corollary}

\begin{proof}
We first consider the case where $\ell_{0}=1$. By definition of the tensor
product multiplicities in (\ref{def_f}) we have $s_{\kappa}^{\ell}%
=\sum_{\lambda\in P_{+}}f_{\lambda}^{\ell}s_{\lambda}$ but also $s_{\kappa
}^{\ell}=s_{\kappa}\times s_{\kappa}^{\ell-1}=\sum_{\lambda\in P_{+}%
}f_{\lambda/\kappa}^{\ell-1}s_{\lambda}$. Therefore $f_{\lambda}^{\ell
}=f_{\lambda/\kappa}^{\ell-1}$ for any $\ell\geq1$ and any $\lambda\in P_{+}$.
We get
\begin{equation}
\underset{\ell\rightarrow+\infty}{\lim}\frac{f_{\lambda^{(\ell)}}^{\ell-1}%
}{f_{\lambda^{(\ell)}}^{\ell}}=\underset{\ell\rightarrow+\infty}{\lim}%
\frac{f_{\lambda^{(\ell)}}^{\ell-1}}{f_{\lambda^{(\ell)}/\kappa}^{\ell-1}%
}=\frac{1}{\tau^{-\kappa}S_{\kappa}(\tau)} \label{inter}%
\end{equation}
by assertion (ii) of Theorem. Now observe that for any $\ell_{0}\geq1$ we
have
\[
\frac{f_{\lambda^{(\ell)}}^{\ell-\ell_{0}}}{f_{\lambda^{(\ell)}}^{\ell}}%
=\frac{f_{\lambda^{(\ell)}}^{\ell-\ell_{0}}}{f_{\lambda^{(\ell)}}^{\ell
-\ell_{0}+1}}\times\cdots\times\frac{f_{\lambda^{(\ell)}}^{\ell-1}}%
{f_{\lambda^{(\ell)}}^{\ell}}.
\]
By using (\ref{inter}) each component of the previous product tends to
$\frac{1}{\tau^{-\kappa}S_{\kappa}(\tau)}$ when $\ell$ tends to infinity which
gives the desired limit.
\end{proof}

\bigskip

The previous theorem also implies that the drift $\overline{m}$ determines the
probability distribution on $B(\pi_{\kappa})$. More precisely, consider $p$
and $p^{\prime}$ two probability distributions defined on $B(\pi_{\kappa})$
from $\tau\in]0,1[^{n}$ and $\tau^{\prime}\in]0,1[^{n}$, respectively. Set
$m=\sum_{\pi\in B(\pi_{\kappa})}p_{\pi}\pi$ and $m^{\prime}=\sum_{\pi\in
B(\pi_{\kappa})}p_{\pi}^{\prime}\pi$.

\begin{proposition}
\label{Prop_F_assymp}We have $\overline{m}=\overline{m}^{\prime}$ if and only
if $\tau=\tau^{\prime}$. Therefore, the map which associates to any $\tau
\in]0,1[^{n}$ the drift $\overline{m}\in\mathcal{D}_{\kappa}$ is a one-to-one correspondence.
\end{proposition}

\begin{proof}
Assume $\overline{m}=\overline{m}^{\prime}$.\ By applying assertion (i) of
Theorem \ref{Th_f_assymp}, we get $\tau^{\gamma}=(\tau^{\prime})^{\gamma}$ for
any $\gamma\in P$.\ Consider $i\in\{1,\ldots,n\}$. For $\gamma=\alpha_{i}$, we
obtain $\tau_{i}=\tau_{i}^{\prime}$. Therefore $\tau=\tau^{\prime}$.
\end{proof}

\section{Some Limit theorems for the Pitman process}

\subsection{The law of large numbers and the central limit theorem for
${\mathcal{W}}$}

We start by establishing two classical limit theorems for $\mathcal{W}$,
deduced from the law of large numbers and the central limit theorem for the
random walk $W=(W_{\ell})_{\ell\geq1}=(X_{1}+\cdots+X_{\ell})_{\ell\geq1}$.
Recall that $m=\sum_{\pi\in B(\pi_{\kappa})}p_{\pi}\pi$ and $\overline
{m}=m(1)$.\ Write $m^{\otimes\infty}$ for the random path such that
\[
m^{\otimes\infty}(t)=\ell\overline{m}+m(t-\ell)\text{ for any }t>0
\]
where $\ell=\left\lfloor t\right\rfloor .\ $

Let $\displaystyle\Gamma=(\Gamma_{i,j})_{1\leq i,j\leq n}=\ ^{t}X_{\ell}\cdot
X_{\ell}$ be the common covariance matrix of each random variable $X_{\ell}$.

\begin{theorem}
\label{ThLLN}Let ${\mathcal{W}}$ be a random path defined on $(B(\pi_{\kappa
})^{\otimes\mathbb{Z}_{\geq0}},p^{\otimes{\mathbb{Z}_{\geq0}}})$ with drift
path $m$. Then, we have
\[
\lim_{\ell\rightarrow+\infty}\frac{1}{\ell}\sup_{t\in\lbrack0,\ell]}\left\Vert
{\mathcal{W}}(t)-m^{\otimes\infty}(t)\right\Vert =0\text{ almost surely.}%
\]
\bigskip Furthermore, the family of random variables $\displaystyle\left(
{\frac{{\mathcal{W}}(t)-m^{\otimes\infty}(t)}{\sqrt{t}}}\right)  _{t>0}$
converges in law as $t\rightarrow+\infty$ towards a centered Gaussian law
$\mathcal{N}(0,\Gamma)$.

{More precisely, setting $\displaystyle{\mathcal{W}^{(\ell)}}(t):={\frac
{{\mathcal{W}}(\ell t)-m^{\otimes\infty}(\ell t)}{\sqrt{\ell}}}$ for any
$0\leq t\leq1$ and $\ell\geq1$, the sequence of random processes
$(\mathcal{W}^{(\ell)}(t))_{\ell\geq1}$ converges to a $n$-dimensional
Brownian motion $(B_{\Gamma}(t))_{0\leq t\leq1}$ with covariance matrix
$\Gamma$. }
\end{theorem}

\begin{proof}
Fix $\ell\geq1$ and observe that
\[
\sup_{t\in\lbrack0,\ell]}\left\Vert {\mathcal{W}}(t)-m^{\otimes\infty
}(t)\right\Vert =\sup_{0\leq k\leq\ell-1}\sup_{t\in\lbrack k,k+1]}\left\Vert
{\mathcal{W}}(t)-k\overline{m}-m(t-k)\right\Vert .
\]
For any $0\leq k\leq\ell$ and $t\in\lbrack k,k+1]$, we have ${\mathcal{W}%
}(t)=W_{k}+X_{k+1}(t-k)$ so that
\begin{equation}
{\mathcal{W}}(t)-m^{\otimes\infty}(t)=W_{k}-k\overline{m}\quad+\quad
\bigl(X_{k+1}(t-k)-m(t-k)\bigr) \label{decomposition}%
\end{equation}
with $\displaystyle\sup_{t\in\lbrack k,k+1]}\left\Vert X_{k+1}%
(t-k)-m(t-k)\right\Vert =\sup_{t\in\lbrack0,1]}\left\Vert X_{k+1}%
(t)-m(t)\right\Vert \leq+2L$, since both paths in $B(\kappa)$ and $m$ have
length $L$, by Proposition \ref{Prop_Length}. It readily follows that
\begin{equation}
\sup_{t\in\lbrack0,\ell]}\left\Vert {\mathcal{W}}(t)-m^{\otimes\infty
}(t)\right\Vert \leq\sup_{0\leq k\leq\ell}\Vert W_{k}-k\overline{m}\Vert+2L.
\label{majoration2L}%
\end{equation}
By the law of large number for the random walk $W=(W_{k})_{k\geq1}$, one gets
$\underset{k\rightarrow+\infty}{\lim}\frac{1}{k}\left\Vert W_{k}-k\overline
{m}\right\Vert =0$ almost surely; this readily implies%
\[
\lim_{\ell\rightarrow+\infty}\frac{1}{\ell}\sup_{t\in\lbrack0,\ell]}\left\Vert
{\mathcal{W}}(t)-m^{\otimes\infty}(t)\right\Vert =0\text{ almost surely.}%
\]

Let us now prove the central limit theorem; for any $t>0$, set $k_{t}%
:=\left\lfloor t\right\rfloor $ and notice that decomposition
(\ref{decomposition}) yields to
\begin{equation}
{\frac{\mathcal{W}(t)-m^{\otimes\infty}(t)}{\sqrt{t}}}=\sqrt{\frac{k_{t}}{t}%
}\times{\frac{W_{k_{t}}-k_{t}\overline{m}}{\sqrt{k_{t}}}}\quad+\quad
{\frac{X_{k_{t}+1}(t-k_{t})-m(t-k_{t})}{\sqrt{t}}}%
\end{equation}
By the central limit theorem in $\mathbb{R}^{n}$, one knows that the sequence
of random variables $\displaystyle\left(  {\frac{W_{k}-k\overline{m}}{\sqrt
{k}}}\right)  _{k\geq1}$ converges in law as $k\rightarrow+\infty$ towards a
centered Gaussian law $\mathcal{N}(0,\Gamma)$; on the other hand, one gets
$\displaystyle\lim_{t\rightarrow+\infty}\sqrt{\frac{k_{t}}{t}}=1$ and
$\displaystyle\limsup_{t\rightarrow+\infty}\Big\Vert{\frac{X_{k_{t}+1}%
(t-k_{t})-m(t-k_{t})}{\sqrt{t}}}\Big\Vert\leq\limsup_{t\rightarrow+\infty
}{\frac{2L}{\sqrt{t}}}=0$, so one may conclude using Slutsky theorem.

{The convergence of the sequence $(\mathcal{W}_{\ell}(t))_{\ell\geq1}$ towards
a Brownian motion goes along the same line One {sets
\[
W^{(\ell)}(t):={\frac{W_{\left\lfloor \ell t\right\rfloor }+(\ell
t-\left\lfloor \ell t\right\rfloor )X_{\left\lfloor \ell t\right\rfloor
+1}(1)-\ell t\overline{m}}{\sqrt{\ell}}}\qquad\mbox{\rm for all}\ \ell
\geq1\ \mbox{\rm and}\ 0\leq t\leq1
\]
and observes that $\displaystyle\Bigl\Vert\mathcal{W}^{(\ell)}(t)-{W}^{(\ell
)}(t)\Bigr\Vert\leq{\frac{2}{\sqrt{\ell}}}(\Vert m\Vert_{\infty}+\Vert
X_{\left\lfloor nt\right\rfloor +1}\Vert_{\infty})$.}}
\end{proof}

\subsection{The law of large numbers and the central limit theorem for
$\mathcal{H}$}

To prove the law of large numbers and the central limit theorem for
$\mathcal{H}$, we need the two following preparatory lemmas. Consider a simple
root $\alpha$ and a trajectory $\eta\in\Omega$ such that $\frac{1}{\ell
}\langle\eta(\ell),\alpha^{\vee}\rangle$ converges to a \emph{positive} limit
when $\ell$ tends to infinity.

\begin{lemma}
\label{lemma_restrict}There exists a nonnegative integer $\ell_{0}$ such that
for any $\ell\geq\ell_{0}$
\[
\inf_{t\in\lbrack0,\ell]}\langle\eta(t),\alpha^{\vee}\rangle=\inf_{t\in
\lbrack0,\ell_{0}]}\langle\eta(t),\alpha^{\vee}\rangle.
\]

\end{lemma}

\begin{proof}
Since $\frac{1}{\ell}\langle\eta(\ell),\alpha^{\vee}\rangle$ converges to a
positive limit, we have in particular that $\underset{\ell\rightarrow+\infty
}{\lim}\langle\eta(\ell),\alpha^{\vee}\rangle=+\infty$. Consider $t>0$ and set
$\ell=\left\lfloor t\right\rfloor $. We can write by definition of $\eta
\in\Omega,$ $\eta(t)=\eta(\ell)+\pi(t-\ell)$ where $\pi$ is a path of
$B(\pi_{\kappa})$. So $\langle\eta(t),\alpha^{\vee}\rangle=\langle\eta
(\ell),\alpha^{\vee}\rangle+\langle\pi(t-\ell),\alpha^{\vee}\rangle$. Since
$\pi\in B(\pi_{\kappa})$, we have
\[
\left\Vert \pi(t-\ell)\right\Vert \leq L
\]
where $L$ is the common length of the paths in $B(\pi_{\kappa})$. So the
possible values of $\langle\pi(t-\ell),\alpha^{\vee}\rangle$ are bounded.
Since $\underset{\ell\rightarrow+\infty}{\lim}\langle\eta(\ell),\alpha^{\vee
}\rangle=+\infty,$ we also get $\lim_{t\rightarrow+\infty}\langle
\eta(t),\alpha^{\vee}\rangle=+\infty$. Recall that $\eta(0)=0$. Therefore
$\inf_{t\in\lbrack0,\ell]}\langle\eta(t),\alpha^{\vee}\rangle\leq0$.\ Since
$\lim_{t\rightarrow+\infty}\langle\eta(t),\alpha^{\vee}\rangle=+\infty$ and
the path $\eta$ is continuous, there should exist an integer $\ell_{0}$ such
that $\inf_{t\in\lbrack0,\ell_{0}]}\langle\eta(t),\alpha^{\vee}\rangle
=\inf_{t\in\lbrack0,\ell_{0}]}\langle\eta(t),\alpha^{\vee}\rangle$ for any
$\ell\geq\ell_{0}$.
\end{proof}

\begin{lemma}
\label{Lem_rec} \ 

\begin{enumerate}
\item[(i)] Consider a simple root $\alpha$ and a trajectory $\eta\in\Omega$
such that $\frac{1}{\ell}\langle\eta(\ell),\alpha^{\vee}\rangle$ converges to
a \emph{positive} limit when $\ell$ tends to infinity. We have for any simple
root $\alpha$%
\[
\sup_{t\in\lbrack0,+\infty\lbrack}\left\Vert {\mathcal{P}}_{\alpha}%
(\eta)(t)-\eta(t)\right\Vert <+\infty
\]
in particular, $\frac{1}{\ell}\langle${$\mathcal{P}$}$_{\alpha}(\eta
)(\ell),\alpha^{\vee}\rangle$ also converges to a positive limit.

\item[(ii)] More generally, let $\alpha_{i_{1}},\cdots,\alpha_{i_{r}},r\geq1,$
be simple roots of $\mathfrak{g}$ and $\eta$ a path in $\Omega$ satisfying
$\displaystyle\lim_{t\rightarrow+\infty}\langle\eta(t),\alpha_{i_{j}}^{\vee
}\rangle=+\infty$ for $1\leq j\leq r$. One gets
\[
\sup_{t\in\lbrack0,+\infty\lbrack}\Vert\mathcal{P}_{\alpha_{i_{1}}}%
\cdots\mathcal{P}_{\alpha_{i_{r}}}(\eta)(t)-\eta(t)\Vert<+\infty.
\]

\end{enumerate}
\end{lemma}

\begin{proof}
(i) By definition of the transform {$\mathcal{P}$}$_{\alpha}$, we have
$\left\Vert {\mathcal{P}}_{\alpha}(\eta)(t)-\eta(t)\right\Vert =\left\vert
\inf_{t\in\lbrack0,t]}\langle\eta(s),\alpha^{\vee}\rangle\right\vert
\left\Vert \alpha^{\vee}\right\Vert $ for any $t\geq0$. By the previous lemma,
there exists an integer $\ell_{0}$ such that for any $t\geq\ell_{0},$
$\left\Vert {\mathcal{P}}_{\alpha}(\eta)(t)-\eta(t)\right\Vert =\left\vert
\inf_{s\in\lbrack0,t]}\langle\eta(s),\alpha^{\vee}\rangle\right\vert
\left\Vert \alpha^{\vee}\right\Vert =\left\vert \inf_{s\in\lbrack0,\ell_{0}%
]}\langle\eta(s),\alpha^{\vee}\rangle\right\vert \left\Vert \alpha^{\vee
}\right\Vert $.\ Since the infimum $\inf_{s\in\lbrack0,\ell_{0}]}\langle
\eta(s),\alpha^{\vee}\rangle$ does not depend on $\ell$, we are done. Now
$\frac{1}{\ell}\langle${$\mathcal{P}$}$_{\alpha}(\eta(\ell)),\alpha^{\vee
}\rangle$ and $\frac{1}{\ell}\langle\eta(\ell),\alpha^{\vee}\rangle$ admit the
same limit.

(ii) Consider $a\in\{2,\ldots,r-1\}$ and assume by induction that we have%
\begin{equation}
\sup_{t\in\lbrack0,+\infty\lbrack}\left\Vert {\mathcal{P}}_{\alpha_{i_{a}}%
}\cdots{\mathcal{P}}_{\alpha_{i_{r}}}(\eta)(t)-m^{\otimes\infty}(t)\right\Vert
<+\infty. \label{HR}%
\end{equation}
We then deduce
\begin{equation}
\lim_{\ell\rightarrow+\infty}\frac{1}{\ell}\langle{\mathcal{P}}_{\alpha
_{i_{a}}}\cdots{\mathcal{P}}_{\alpha_{i_{r}}}(\eta)(\ell),\alpha_{i_{a-1}%
}^{\vee}\rangle=\langle\overline{m},\alpha_{i_{a-1}}^{\vee}\rangle>0.
\label{TU}%
\end{equation}
This permits to apply Lemma \ref{Lem_rec} with $\eta^{\prime}=${$\mathcal{P}$%
}$_{\alpha_{i_{a}}}\cdots${$\mathcal{P}$}$_{\alpha_{i_{r}}}(\eta)$ and
$\alpha=\alpha_{i_{a-1}}$. We get%
\[
\sup_{t\in\lbrack0,+\infty\lbrack}\left\Vert {\mathcal{P}}_{\alpha_{i_{a-1}}%
}\cdots{\mathcal{P}}_{\alpha_{i_{r}}}(\eta)(t)-{\mathcal{P}}_{\alpha_{i_{a}}%
}\cdots{\mathcal{P}}_{\alpha_{i_{r}}}(\eta)(t)\right\Vert <+\infty\text{.}%
\]
By using (\ref{HR}), this gives%
\begin{equation}
\sup_{t\in\lbrack0,+\infty\lbrack}\left\Vert {\mathcal{P}}_{\alpha_{i_{a-1}}%
}\cdots{\mathcal{P}}_{\alpha_{i_{r}}}(\eta)(t)-m^{\otimes\infty}(t)\right\Vert
<+\infty. \label{HR*}%
\end{equation}
We thus have proved by induction that (\ref{HR*}) holds for any $a=2,\ldots
,r-1$.
\end{proof}

\begin{theorem}
\label{ThLLN P}Let ${\mathcal{W}}$ be a random path defined on $\Omega
=(B(\pi_{\kappa})^{\otimes\mathbb{Z}_{\geq0}},p^{\otimes{\mathbb{Z}_{\geq0}}%
})$ with drift path $m$ and let ${\mathcal{H}}={\mathcal{P}}({\mathcal{W)}}$
be its Pitman transform. Assume $\overline{m}\in\mathcal{D}_{\kappa}$.\ Then,
we have
\[
\lim_{\ell\rightarrow+\infty}\frac{1}{\ell}\sup_{t\in\lbrack0,\ell]}\left\Vert
{\mathcal{H}}(t)-m^{\otimes\infty}(t)\right\Vert =0\text{ almost surely.}%
\]
Furthermore, the family of random variables $\displaystyle\left(
{\frac{{\mathcal{H}}(t)-m^{\otimes\infty}(t)}{\sqrt{t}}}\right)  _{t>0}$
converges in law as $t\rightarrow+\infty$ towards a centered Gaussian law
$\mathcal{N}(0,\Gamma)$.
\end{theorem}

\begin{proof}
Recall we have {$\mathcal{P}$}$=${$\mathcal{P}$}$_{\alpha_{i_{1}}}\cdots
${$\mathcal{P}$}$_{\alpha_{i_{r}}}$ by Proposition (\cite{BBO}). Consequently,
by Theorem \ref{ThLLN} and Lemma \ref{Lem_rec}, the random variable
$\mathcal{H}-\mathcal{W}=\mathcal{P}(\mathcal{W})-\mathcal{W}$ is finite
almost surely. It follows that
\[
\limsup_{\ell\rightarrow+\infty}{\frac{1}{\ell}}\sup_{t\in\lbrack0,\ell
]}\left\Vert {\mathcal{H}}(t)-m^{\otimes l}(t)\right\Vert \leq\limsup
_{\ell\rightarrow+\infty}{\frac{1}{\ell}}\sup_{t\in\lbrack0,\ell]}\left\Vert
{\mathcal{W}}(t)-m^{\otimes l}(t)\right\Vert +\limsup_{\ell\rightarrow+\infty
}{\frac{1}{\ell}}\sup_{t\geq0}\Vert\mathcal{H}(t)-\mathcal{W}(t)\Vert=0
\]
almost surely. To get the central limit theorem for the process $\mathcal{H}%
(t)$, we write similarly
\[
{\frac{{\mathcal{H}}(t)-m^{\otimes l}(t)}{\sqrt{t}}}={\frac{{\mathcal{W}%
}(t)-m^{\otimes l}(t)}{\sqrt{t}}}+{\frac{{\mathcal{H}}(t)-{\mathcal{W}}%
(t)}{\sqrt{t}}}.
\]
By Theorem \ref{ThLLN}, the first term in this decomposition satisfies the
central limit theorem; on the other hand the second one tends to $0$ almost
surely and one concludes using Slutsky theorem
\end{proof}

\section{Harmonic functions on multiplicative graphs associated to a central
measure}

Harmonic functions on the Young lattice are the key ingredients in the study
of the asymptotic representation theory of the symmetric group. In fact, it
was shown by Kerov and Vershik that these harmonic functions completely
determine the asymptotic characters of the symmetric groups. We refer the
reader to \cite{Ker} for a detailed review. The Young lattice is an oriented
graph with set of vertices the set of all partitions (each partition is
conveniently identified its Young diagram). We have an arrow $\lambda
\rightarrow\Lambda$ between the partitions $\lambda$ and $\Lambda$ when
$\Lambda$ can be obtained by adding a box to $\lambda$. The Young lattice is
an example of branching graph in the sense that its structure reflects the
branching rules between the representations theory of the groups $S_{\ell}$
and $S_{\ell+1}$ with $\ell>0$. One can also consider harmonic functions on
other interesting graphs.

Here we focus on graphs defined from the weight lattice of $\mathfrak{g}$.
These graphs depend on a fixed $\kappa\in P_{+}$ and are multiplicative in the
sense that a positive integer, equal to a tensor product multiplicity, is
associated to each arrow. We call them the multiplicative tensor graphs. We
are going to associate a Markov chain to each multiplicative tensor graph
$\mathcal{G}$. The aim of this section is to determine the harmonic functions
on $\mathcal{G}$ when this associated Markov chain is assumed to have a drift.
We will show this is equivalent to determine the central probability measure
on the subset $\Omega_{\mathcal{C}}$ containing all the trajectories which
remains in $\mathcal{C}$.\ When $\mathfrak{g}=\mathfrak{sl}_{n+1}$ and
$\kappa=\omega_{1}$ (that is $V(\kappa)=\mathbb{C}^{n+1}$ is the defining
representation of $\mathfrak{sl}_{n+1}$), $\mathcal{G}$ is the subgraph of the
Young lattice obtained by considering only the partitions with at most $n+1$
parts and we recover the harmonic functions as specializations of Schur polynomials.

\subsection{Multiplicative tensor graphs, harmonic functions and central
measures}

\label{subsec_graphs}So assume $\kappa\in P_{+}$ is fixed. We denote by
$\mathcal{G}$ the oriented graph with set of vertices the pairs $(\lambda
,\ell)\in P_{+}\times\mathbb{Z}_{\geq0}$ and arrow
\[
(\lambda,\ell)\overset{m_{\lambda,\kappa}^{\Lambda}}{\rightarrow}(\Lambda
,\ell+1)
\]
with multiplicity $m_{\lambda,\kappa}^{\Lambda}$ when $m_{\lambda,\kappa
}^{\Lambda}>0$. In particular there is no arrows between $(\lambda,\ell)$ and
$(\Lambda,\ell+1)$ when $m_{\kappa,\kappa}^{\Lambda}=0$.

\begin{example}
Consider $\mathfrak{g}=\mathfrak{sp}_{2n}.$ Then $P=\mathbb{Z}^{n}$ and
$P_{+}$ can be identified with the set of partitions with at most $n$ parts.
For $\kappa=\omega_{1}$ the graph $\mathcal{G}$ is such that $(\lambda
,\ell)\rightarrow(\Lambda,\ell+1)$ with $m_{\lambda,\kappa}^{\Lambda}=1$ if
and only of the Young diagram of $\Lambda$ is obtained from that of $\lambda$
by adding or deleting one box. We have drawn below the connected component of
$(\,\emptyset,0\,)$ up to $\ell\leq3$.%

\[%
\begin{tabular}
[c]{ccccccc}
&  & $(\,\emptyset,0\,)$ &  &  &  & \\
&  & $\downarrow$ &  &  &  & \\
&  & $\left(  \,%
\begin{tabular}
[c]{|l|}\hline
\\\hline
\end{tabular}
,1\,\right)  $ &  &  &  & \\
& $%
\begin{array}
[c]{c}%
\\
\swarrow
\end{array}
$ & $%
\begin{array}
[c]{c}%
\\
\downarrow
\end{array}
$ & $%
\begin{array}
[c]{c}%
\\
\searrow
\end{array}
$ &  &  & \\
$(\,\emptyset,2\,)$ &  & $\left(  \,%
\begin{tabular}
[c]{|l|l|}\hline
& \\\hline
\end{tabular}
,2\,\right)  $ &  & $\left(  \,%
\begin{tabular}
[c]{|l|}\hline
\\\hline
\\\hline
\end{tabular}
,2\,\right)  $ &  & \\
& $%
\begin{array}
[c]{c}%
\\
\swarrow\searrow
\end{array}
$ & $%
\begin{array}
[c]{c}%
\\
\downarrow
\end{array}
$ & $%
\begin{array}
[c]{c}%
\\
\swarrow\searrow
\end{array}
$ & $%
\begin{array}
[c]{c}%
\\
\downarrow
\end{array}
$ & $%
\begin{array}
[c]{c}%
\\
\searrow
\end{array}
$ & \\
$\left(  \,%
\begin{tabular}
[c]{|l|l|l|}\hline
&  & \\\hline
\end{tabular}
,3\,\right)  $ &  & $\left(  \,%
\begin{tabular}
[c]{|l|}\hline
\\\hline
\end{tabular}
,3\,\right)  $ &  & $\left(  \,%
\begin{tabular}
[c]{|l|l}\cline{1-1}
& \\\hline
& \multicolumn{1}{|l|}{}\\\hline
\end{tabular}
,3\,\right)  $ &  & $\left(  \,%
\begin{tabular}
[c]{|l|}\hline
\\\hline
\\\hline
\\\hline
\end{tabular}
,3\,\right)  $\\
$\vdots$ &  & $\vdots$ &  & $\vdots$ &  & $\vdots$%
\end{tabular}
\ \
\]

\end{example}

Observe that in the case $\mathfrak{g}=\mathfrak{sl}_{n+1}$ and $\kappa
=\omega_{1},$ we have $m_{\lambda,\kappa}^{\Lambda}=1$ if and only if of the
Young diagram of $\Lambda$ is obtained by adding one box to that of $\lambda$
and $m_{\lambda,\kappa}^{\Lambda}=0$ otherwise.\ So in this very particular
case, it is not useful to keep the second component $\ell$ since it is equal
to the rank of the partition $\lambda$. The vertices of $\mathcal{G}$ are
simply the partitions with at most $n$ parts (i.e. whose Young diagram has at
most $n$ rows).

\bigskip

Now return to the general case. Our aim is now to relate the harmonic
functions on $\mathcal{G}$ and the central probability distributions on the
set $\Omega_{\mathcal{C}}$ of infinite trajectories with steps in
$B(\pi_{\kappa})$ which remain in $\mathcal{C}$.\ We will identify the
elements of $P_{+}\times\mathbb{Z}_{\geq0}$ as elements of the $\mathbb{R}%
$-vector space $P_{\mathbb{R}}\times\mathbb{R}$ (recall $P_{\mathbb{R}%
}=\mathbb{R}^{n}$).\ For any $\ell\geq0$, set $H^{\ell}=\{\pi\in B(\pi
_{\kappa})^{\otimes\ell}\mid\operatorname{Im}\pi\subset\mathcal{C}\}$. Also if
$\lambda\in P_{+}$, set $H_{\lambda}^{\ell}=\{\pi\in H^{\ell}\mid
\mathrm{wt}(\pi)=\lambda\}$. Given $\pi\in H^{\ell}$, we denote by
\[
C_{\pi}=\{\omega\in\Omega_{\mathcal{C}}\mid\Pi_{\ell}(\omega)=\pi\}
\]
the cylinder defined by $\pi$.\ We have $C_{\varnothing}=\Omega_{\mathcal{C}}%
$.\ Each probability distribution $\mathbb{Q}$ on $\Omega_{\mathcal{C}}$ is
determined by its values on the cylinders and we must have
\[
\sum_{\pi\in H^{\ell}}\mathbb{Q}(C_{\pi})=1
\]
for any $\ell\geq0$.

\begin{definition}
A central probability distribution on $\Omega_{\mathcal{C}}$ is a probability
distribution $\mathbb{Q}$ on $\Omega_{\mathcal{C}}$ such that $\mathbb{Q}%
(C_{\pi})=\mathbb{Q}(C_{\pi^{\prime}})$ provided that $\mathrm{wt}%
(\pi)=\mathrm{wt}(\pi^{\prime})$ and $\pi,\pi^{\prime}$ \emph{have the same
length}.
\end{definition}

Consider a central probability distribution $\mathbb{Q}$ on $\Omega
_{\mathcal{C}}$. For any $\ell\geq1$, we have $\sum_{\pi\in H^{\ell}%
}\mathbb{Q(}C_{\pi})=1$, so it is possible to define a probability
distribution $q$ on $H^{\ell}$ by setting $q_{\pi}=\mathbb{Q}(C_{\pi})$ for
any $\pi\in H^{\ell}$. Since $\mathbb{Q}$ is central, we can also define the
function
\begin{equation}
\varphi:\left\{
\begin{array}
[c]{c}%
\mathcal{G}\rightarrow\lbrack0,1]\\
(\lambda,\ell)\longmapsto\mathbb{Q}(C_{\pi})
\end{array}
\right.  \label{def_phi}%
\end{equation}
where $\pi$ is any path of $H^{\ell}$. Now observe that%
\[
C_{\pi}=\bigsqcup\limits_{\eta\in B(\pi_{\kappa})\mid\operatorname{Im}%
(\pi\otimes\eta)\subset\mathcal{C}}C_{\pi\otimes\eta}.
\]
This gives
\begin{equation}
\mathbb{Q}(C_{\pi})=\sum_{\eta\in B(\pi_{\kappa})\mid\operatorname{Im}%
(\pi\otimes\eta)\subset\mathcal{C}}\mathbb{Q}(C_{\pi\otimes\eta}).
\label{Q-*_prob_dist}%
\end{equation}
Assume $\pi\in H_{\lambda}^{\ell}$.\ By Theorem \ref{Th_Littel}, the
cardinality of the set $\{\eta\in B(\pi_{\kappa})\mid\operatorname{Im}%
(\pi\otimes\eta)\subset\mathcal{C}$ and $\mathrm{wt}(\pi\otimes\eta
)=\Lambda\}$ is equal to $m_{\lambda,\kappa}^{\Lambda}$ .\ Therefore, we get
\begin{equation}
\varphi(\lambda,\ell)=\sum_{\Lambda}m_{\lambda,\kappa}^{\Lambda}%
\varphi(\Lambda,\ell+1). \label{harm}%
\end{equation}
This means that the function $\varphi$ is \emph{harmonic} on the
multiplicative graph $\mathcal{G}$.

Conversely, if $\varphi^{\prime}$ is harmonic on the multiplicative graph
$\mathcal{G}$, for any cylinder $C_{\pi}$ in $\Omega_{\mathcal{C}}$ with
$\pi\in H_{\lambda}^{\ell}$, we set $\mathbb{Q}^{\prime}(C_{\pi}%
)=\varphi^{\prime}(\lambda,\ell)$. Then $\mathbb{Q}^{\prime}$ is a probability
distribution on $\Omega_{\mathcal{C}}$ since it verifies (\ref{Q-*_prob_dist})
and is clearly central. Therefore, \emph{a central probability distribution on
}$\Omega_{\mathcal{C}}$\emph{ is characterized by its associated harmonic
function} defined by (\ref{def_phi}).

\subsection{Harmonic function on a multiplicative tensor graph}

Let $\mathbb{Q}$ a central probability distribution on $\Omega_{\mathcal{C}}%
$.\ Consider $\pi=\pi_{1}\otimes\cdots\otimes\pi_{\ell}\in H_{\lambda}^{\ell}$
and $\pi^{\#}=\pi_{1}\otimes\cdots\otimes\pi_{\ell}\otimes\pi_{\ell+1}\in
H_{\Lambda}^{\ell+1}$.\ Since we have the inclusion of events $C_{\pi^{\#}%
}\subset C_{\pi}$, we get%
\[
\mathbb{Q}(C_{\pi^{\#}}\mid C_{\pi})=\frac{\mathbb{Q}(C_{\pi^{\#}}%
)}{\mathbb{Q}(C_{\pi})}=\frac{\varphi(\Lambda,\ell+1)}{\varphi(\lambda,\ell)}%
\]
where the last equality is by definition of the harmonic function $\varphi$
(which exists since $\mathbb{Q}$ is central).\ Let us emphasize that
$\mathbb{Q}(C_{\pi^{\#}})$ and $\mathbb{Q}(C_{\pi})$ do not depend on the
paths $\pi$ and $\pi^{\#}$ but only on their lengths and their ends $\lambda$
and $\Lambda$. We then define a Markov chain $Z=(Z_{\ell})_{\ell\geq0}$ from
$(\Omega_{\mathcal{C}},\mathbb{Q})$ with values in $\mathcal{G}$ {and starting
at} $Z_{0}=(\mu,\ell_{0})\in\mathcal{G}$ by
\[
Z_{\ell}(\omega)=(\mu+\omega(\ell),\ell+\ell_{0}).
\]
Its transition probabilities are given by%
\[
\Pi_{Z}((\lambda,\ell),(\Lambda,\ell+1))=\sum_{\pi^{\#}}\mathbb{Q}(C_{\pi
^{\#}}\mid C_{\pi})
\]
where $\pi$ is any path in $H_{\lambda}^{\ell}$ and the sum runs over all the
paths $\pi^{\#}\in H_{\Lambda}^{\ell+1}$ such that $\pi^{\#}=\pi\otimes
\pi_{\ell+1}$.\ Observe, the above sum does not depend on the choice of $\pi$
in $H_{\lambda}^{\ell}$ because $\mathbb{Q}$ is central. Since there are
$m_{\lambda,\kappa}^{\Lambda}$ such pairs, we get%
\begin{equation}
\Pi_{Z}((\lambda,\ell),(\Lambda,\ell+1))=\frac{m_{\lambda,\kappa}^{\Lambda
}\varphi(\Lambda,\ell+1)}{\varphi(\lambda,\ell)}\label{PiZ}%
\end{equation}
and by (\ref{harm}) $Z=(Z_{\ell})_{\ell\geq0}$ is indeed a Markov chain.\ We
then write $\mathbb{Q}_{(\mu,\ell_{0})}(Z_{\ell}=(\lambda,\ell))$ for the
probability that $Z_{\ell}=(\lambda,\ell)$ when the initial value is
$Z_{0}=(\mu,\ell_{0})$. When $Z_{0}=(0,0)$, we simply write $\mathbb{Q}%
(Z_{\ell}=(\lambda,\ell))=\mathbb{Q}_{(0,0)}(Z_{\ell}=(\lambda,\ell))$.

\begin{lemma}
\label{Lem_Qmu}For any $\mu,\lambda\in P_{+}$ and any integer $\ell_{0}\geq1$,
we have
\[
\mathbb{Q}_{(\mu,\ell_{0})}(Z_{\ell-\ell_{0}}=(\lambda,\ell))=f_{\lambda/\mu
}^{(\ell-\ell_{0})}\frac{\varphi(\lambda,\ell)}{\varphi(\mu,\ell_{0})} \text{
for any }\ell\geq\ell_{0}.
\]

\end{lemma}

\begin{proof}
By (\ref{PiZ}), the probability $\mathbb{Q}_{(\mu,\ell_{0})}(Z_{\ell-\ell_{0}%
}=(\lambda,\ell))$ is equal to the quotient $\frac{\varphi(\lambda,\ell
)}{\varphi(\mu,\ell_{0})}$ times the number of paths in $\mathcal{C}$ of
length $\ell-\ell_{0}$ starting at $\mu$ and ending at $\lambda$.\ The lemma
thus follows from the fact that this number is equal to $f_{\lambda/\mu
}^{(\ell-\ell_{0})}$ by Theorem \ref{Th_Littel}.
\end{proof}

\bigskip

We will say that the family of Markov chains $Z$ with transition probabilities
given by (\ref{PiZ}) and initial distributions of the form $Z_{0}=(\mu
,\ell_{0})\in\mathcal{G}$ admits a drift $\overline{m}\in P_{\mathbb{R}}$ when%
\[
\lim_{\ell\rightarrow+\infty}\frac{Z_{\ell}}{\ell}=(\overline{m},1)\text{
}\mathbb{Q}\text{-almost surely}%
\]
for any initial distributions $Z_{0}=(\mu,\ell_{0})\in\mathcal{G}$.

\begin{theorem}
\label{Th_central G}Let $\mathbb{Q}$ be a central probability distribution on
$\Omega_{\mathcal{C}}$ such that $Z$ admits the drift $\overline{m}%
\in\mathcal{D}_{\kappa}$ (see (\ref{DefDkappa})).

\begin{enumerate}
\item[(i)] The associated harmonic function $\varphi$ on $\mathcal{G}$
verifies $\varphi(\mu,\ell_{0})=\frac{\tau^{-\mu}S_{\mu}(\tau)}{\tau
^{-\ell_{0}\kappa}S_{\kappa}^{\ell_{0}}(\tau)}$for any $\mu\in P_{+}$ and any
$\ell_{0}\geq0$ where $\tau$ is determined by $\overline{m}$ as prescribed by
Proposition \ref{Prop_F_assymp}.

\item[(ii)] The probability transitions (\ref{PiZ}) do not depend on $\ell$.

\item[(iii)] For any $\pi\in H_{\mu}^{\ell_{0}},$ we have $\mathbb{Q}(C_{\pi
})=\frac{\tau^{-\mu}S_{\mu}(\tau)}{\tau^{-\ell_{0}\kappa}S_{\kappa}^{\ell_{0}%
}(\tau)}$. In particular, $\mathbb{Q}$ is the unique central probability
distribution on $\Omega_{\mathcal{C}}$ such that $Z$ admits the drift
$\overline{m}$. We will denote it by $\mathbb{Q}_{\overline{m}}$.
\end{enumerate}
\end{theorem}

\begin{proof}
(i). Consider a sequence of random dominant weights of the form $\lambda
^{(\ell)}=\ell\overline{m}+o(\ell)$.\ We get by using Lemma \ref{Lem_Qmu}%
\begin{multline*}
\frac{f_{\lambda^{(\ell)}/\mu}^{(\ell-\ell_{0})}}{f_{\lambda^{(\ell)}}%
^{(\ell)}}\times\frac{1}{\varphi(\mu,\ell_{0})}=f_{\lambda^{(\ell)}/\mu
}^{(\ell-\ell_{0})}\times\frac{\varphi(\lambda^{(\ell)},\ell)}{\varphi
(\mu,\ell_{0})}\times\lbrack f_{\lambda^{(\ell)}}^{(\ell)}\times
\varphi(\lambda^{(\ell)},\ell)]^{-1}\\
=\frac{\mathbb{Q}_{(\mu,\ell_{0})}(Z_{\ell-\ell_{0}}=(\lambda^{(\ell)},\ell
))}{\mathbb{Q}(Z_{\ell}=(\lambda^{(\ell)},\ell))}=\frac{\mathbb{Q}_{(\mu
,\ell_{0})}(\frac{Z_{\ell-\ell_{0}}}{\ell-\ell_{0}}=(\frac{\lambda^{(\ell)}%
}{\ell-\ell_{0}},\frac{\ell}{\ell-\ell_{0}}))}{\mathbb{Q}(\frac{Z_{\ell}}%
{\ell}=(\frac{\lambda^{(\ell)}}{\ell},1))}.
\end{multline*}
Since $Z$ admits the drift $\overline{m}$, we obtain
\[
\lim_{\ell\rightarrow+\infty}\frac{\mathbb{Q}_{(\mu,\ell_{0})}(\frac
{Z_{\ell-\ell_{0}}}{\ell-\ell_{0}}=(\frac{\lambda^{(\ell)}}{\ell-\ell_{0}%
},\frac{\ell}{\ell-\ell_{0}}))}{\mathbb{Q}(\frac{Z_{\ell}}{\ell}%
=(\frac{\lambda^{(\ell)}}{\ell},1))}=\frac{1}{1}=1\text{ and }\lim
_{\ell\rightarrow+\infty}\frac{f_{\lambda^{(\ell)}/\mu}^{(\ell-\ell_{0})}%
}{f_{\lambda^{(\ell)}}^{(\ell)}}\times\frac{1}{\varphi(\mu,\ell_{0}%
)}=1\text{.}%
\]
This means that
\[
\varphi(\mu,\ell_{0})=\lim_{\ell\rightarrow+\infty}\frac{f_{\lambda^{(\ell
)}/\mu}^{(\ell-\ell_{0})}}{f_{\lambda^{(\ell)}}^{(\ell)}}.
\]
Now by Theorem \ref{Th_f_assymp} and since $\overline{m}\in\mathcal{D}%
_{\kappa}$ we can write
\[
\lim_{\ell\rightarrow+\infty}\frac{f_{\lambda^{(\ell)}/\mu}^{(\ell-\ell_{0})}%
}{f_{\lambda^{(\ell)}}^{(\ell)}}=\lim_{\ell\rightarrow+\infty}\frac
{f_{\lambda^{(\ell)}/\mu}^{(\ell-\ell_{0})}}{f_{\lambda^{(\ell)}}^{(\ell
-\ell_{0})}}\times\lim_{\ell\rightarrow+\infty}\frac{f_{\lambda^{(\ell)}%
}^{(\ell-\ell_{0})}}{f_{\lambda^{(\ell)}}^{(\ell)}}=\frac{\tau^{-\mu}S_{\mu
}(\tau)}{\tau^{-\ell_{0}\kappa}S_{\kappa}^{\ell_{0}}(\tau)}%
\]
where $\tau\in]0,1[^{n}$ is determined by the drift $\overline{m}$ as
prescribed by Proposition \ref{Prop_F_assymp}. We thus obtain $\varphi
(\mu,\ell_{0})=\frac{\tau^{-\mu}S_{\mu}(\tau)}{\tau^{-\ell_{0}\kappa}%
S_{\kappa}^{\ell_{0}}(\tau)}$.

(ii). We have $\Pi_{Z}((\lambda,\ell),(\Lambda,\ell+1))=\frac{m_{\lambda
,\kappa}^{\Lambda}\varphi(\Lambda,\ell+1)}{\varphi(\lambda,\ell)}%
=\frac{S_{\Lambda}(\tau)}{S_{\kappa}(\tau)S_{\lambda}(\tau)}\tau
^{\kappa+\lambda-\Lambda}m_{\lambda,\kappa}^{\Lambda}$ which does not depend
on $\ell$.

(iii). This follows from the fact that $\mathbb{Q}(C_{\pi})=\varphi
(\lambda,\ell)$ for any $\pi\in H_{\lambda}^{\ell}$.
\end{proof}

\bigskip

Consider $\overline{m}\in\mathcal{D}_{\kappa}$ and write $\tau$ for the
corresponding $n$-tuple in $]0,1[^{n}$. Let $W$ be the random walk starting at
$0$ defined on $P$ from $\kappa$ and $\tau$ as in \S \ \ref{Subsec-RPW}.

\begin{corollary}
Let $\mathbb{Q}$ be a central probability distribution on $\Omega
_{\mathcal{C}}$ such that $Z$ admits the drift $\overline{m}\in\mathcal{D}%
_{\kappa}$.\ {Then, the processes $(Z_{\ell})_{\ell}$ and $(({\mathcal{P}%
}(W_{\ell}),\ell))_{\ell}$ have the same law.}
\end{corollary}

\begin{proof}
By the previous theorem, the transitions of the Markov chain $Z$ on
$\mathcal{G}$ are given by $\Pi_{Z}((\lambda,\ell),(\Lambda,\ell
+1))=\frac{m_{\lambda,\kappa}^{\Lambda}\varphi(\Lambda,\ell+1)}{\varphi
(\lambda,\ell)}$. By Theorem \ref{Th_LawH}, the transition matrix $\Pi_{Z}$
thus coincides with the transition matrix of {$\mathcal{P}$}$(W)$ as desired.
\end{proof}

\bigskip

Let $\mathbb{P}_{\overline{m}}$ and $\mathbb{Q}_{\overline{m}}$ be the
probability distributions associated to $\overline{m}$ (recall $\overline{m}$
determines $\tau\in]0,1[^{n}$) defined on the spaces $\Omega$ and
$\Omega_{\mathcal{C}}$, respectively.

\begin{corollary}
\label{Cor-Iso-ProbaSpace}The Pitman transform $\mathcal{P}$ is a homomorphism
of probability spaces between $(\Omega,\mathbb{P}_{\overline{m}})$ and
$(\Omega_{\mathcal{C}},\mathbb{Q}_{\overline{m}})$, that is we have
\[
\mathbb{Q}_{\overline{m}}\mathbb{(}C_{\pi})=\mathbb{P}_{\overline{m}%
}(\mathcal{P}^{-1}(C_{\pi}))
\]
for any $\ell\geq1$ and any $\pi\in H^{\ell}$.
\end{corollary}

\begin{proof}
Assume $\pi\in H_{\lambda}^{\ell}$.$\ $We have $\mathbb{Q}_{\overline{m}%
}\mathbb{(}C_{\pi})=\varphi(\lambda,\ell)=\frac{\tau^{-\lambda}S_{\lambda
}(\tau)}{\tau^{-\ell\kappa}S_{\kappa}^{\ell}(\tau)}$.\ By definition of the
generalized Pitman transform $\mathcal{P}$, $\mathcal{P}^{-1}(C_{\pi
})=\{\omega\in\Omega\mid\mathcal{P}(\Pi_{\ell}(\omega))=\pi\}$, that is
$\mathcal{P}^{-1}(C_{\pi})$ is the set of all trajectories in $\Omega$ which
remains in the connected component $B(\pi)\subset B(\pi_{\kappa})^{\otimes
\ell}$ for any $t\in\lbrack0,\ell]$.\ We thus have $\mathbb{P}_{\overline{m}%
}(\mathcal{P}^{-1}(C_{\pi}))=p^{\otimes\ell}(B(\pi))=\frac{\tau^{-\lambda
}S_{\lambda}(\tau)}{\tau^{-\ell\kappa}S_{\kappa}^{\ell}(\tau)}$ by assertion
(ii) of Theorem \ref{Th_LawH}. Therefore we get $\mathbb{P}_{\overline{m}%
}(\mathcal{P}^{-1}(C_{\pi}))=\mathbb{Q}_{\overline{m}}\mathbb{(}C_{\pi})$ as desired.
\end{proof}

\section{Isomorphism of dynamical systems}

\label{Sec_Conse}In this section, we first explain how the trajectories in
$\Omega$ and $\Omega_{\mathcal{C}}$ can be equipped with natural shifts $S$
and $J$, respectively. We then prove that the generalized Pitman transform
$\mathcal{P}$ intertwines $S$ and $J$. When $\mathfrak{g}=\mathfrak{sl}_{n+1}$
and $\kappa=\omega_{1}$, we recover in particular some analogue results of
\cite{Snia}.

\subsection{The shift operator}

{Let $S:\Omega\rightarrow\Omega$ be the shift operator on $\Omega$ defined
by}
\[
S(\pi)=S(\pi_{1}\otimes\pi_{2}\otimes\pi_{3}\otimes\cdots):=(\pi_{2}\otimes
\pi_{3}\otimes\ldots)
\]
for any trajectory $\pi=\pi_{1}\otimes\pi_{2}\otimes\pi_{3}\otimes\cdots
\in\Omega$.\ Observe that $S$ is measure preserving for the probability
distribution $\mathbb{P}_{\overline{m}}$.\ We now introduce the map
$J:\Omega_{\mathcal{C}}\rightarrow\Omega_{\mathcal{C}}$ defined by
\[
J(\pi)=\mathcal{P}\circ S(\pi)
\]
for any trajectory $\pi\in\Omega_{\mathcal{C}}$. Observe that $S(\pi)$ does
not belong to $\Omega_{\mathcal{C}}$ in general so we need to apply the Pitman
transform $\mathcal{P}$ to ensure that $J$ takes values in $\Omega
_{\mathcal{C}}$.

\subsection{Isomorphism of dynamical systems}

\begin{theorem}
\ \ \label{Th_PSJ}

\begin{enumerate}
\item[(i)] The Pitman transform is a factor map of dynamical systems, i.e. the
following diagram commutes :%
\[%
\begin{array}
[c]{ccc}%
\Omega & \overset{S}{\rightarrow} & \Omega\\
\mathcal{P}\downarrow\text{ \ } &  & \text{ \ }\downarrow\mathcal{P}\\
\Omega_{\mathcal{C}} & \underset{J}{\rightarrow} & \Omega_{\mathcal{C}}%
\end{array}
\]

\item[(ii)] {For any $\overline{m}\in\mathcal{D}_{\kappa}$, the transformation
$J:\Omega_{\mathcal{C}}\rightarrow\Omega_{\mathcal{C}}$ is measure preserving
with respect to the (unique) central probability distribution $\mathbb{Q}%
_{\overline{m}}$ with drift $\overline{m}$.}
\end{enumerate}
\end{theorem}

\begin{proof}
(i). To prove assertion (i), it suffices to establish that the above diagram
commutes on trajectories of any finite length $\ell>0$. So consider $\pi
=\pi_{1}\otimes\pi_{2}\otimes\cdots\otimes\pi_{\ell}\in B(\pi_{\kappa
})^{\otimes\ell}$ and set $\mathcal{P}(\pi)=\pi_{1}^{+}\otimes\pi_{2}%
^{+}\otimes\cdots\otimes\pi_{\ell}^{+}$. We have to prove that
\[
\mathcal{P}(\pi_{2}\otimes\cdots\otimes\pi_{\ell})=\mathcal{P}(\pi_{2}%
^{+}\otimes\cdots\otimes\pi_{\ell}^{+})
\]
which means that both vertices $\pi_{2}\otimes\cdots\otimes\pi_{\ell}$ and
$\pi_{2}^{+}\otimes\cdots\otimes\pi_{\ell}^{+}$ belong to the same connected
component of $B(\pi_{\kappa})^{\otimes\ell-1}$.\ We know that $\mathcal{P}%
(\pi)$ is the highest weight vertex of $B(\pi)$.\ This implies that there
exists a sequence of root operators $\tilde{e}_{i_{1}},\ldots,\tilde{e}%
_{i_{r}}$ such that
\begin{equation}
\pi_{1}^{+}\otimes\pi_{2}^{+}\otimes\cdots\otimes\pi_{\ell}^{+}=\tilde
{e}_{i_{1}}\cdots\tilde{e}_{i_{r}}(\pi_{1}\otimes\pi_{2}\otimes\cdots
\otimes\pi_{\ell}). \label{eqtens}%
\end{equation}
By (\ref{Ten_Prod}), we can define a subset $X:=\{k\in\{1,\ldots,r\}$ such
that $\tilde{e}_{i_{k}}$ acts on the first component of the tensor product
$\tilde{e}_{i_{k+1}}\cdots\tilde{e}_{i_{r}}(\pi_{1}\otimes\pi_{2}\otimes
\cdots\otimes\pi_{\ell})\}$. We thus obtain
\[
\pi_{2}^{+}\otimes\cdots\otimes\pi_{\ell}^{+}=\prod_{k\in\{1,\ldots
,r\}\setminus X}\tilde{e}_{i_{k}}(\pi_{2}\otimes\cdots\otimes\pi_{\ell})
\]
which shows that $\pi_{2}\otimes\cdots\otimes\pi_{\ell}$ and $\pi_{2}%
^{+}\otimes\cdots\otimes\pi_{\ell}^{+}$ belong to the same connected component
of $B(\pi_{\kappa})^{\otimes\ell-1}$. They thus have the same highest weight
path as desired.

(ii). Let $A\subset\Omega_{\mathcal{C}}$ be a $\mathbb{Q}$-measurable set. We
have $\mathbb{Q}(J^{-1}(A))=\mathbb{P}(\mathcal{P}^{-1}(J^{-1}(A))$ since
$\mathcal{P}$ is an homomorphism. Using the fact that previous diagram
commutes and $S$ preserves $\mathbb{P}$, we get $\mathbb{Q}(J^{-1}%
(A))=\mathbb{P}(S^{-1}(\mathcal{P}^{-1}(A)))=\mathbb{P}(\mathcal{P})$, so that
so $\mathbb{Q}(J^{-1}(A)) =\mathbb{Q}(A)$ since $\mathcal{P}$ is an homomorphism.
\end{proof}

\section{Dual random path and the inverse Pitman transform}

It is well known (see \cite{Pit}) that the Pitman transform on the line is
reversible. The aim of this paragraph is to establish that $\mathcal{E}$,
restricted to a relevant set of infinite trajectories with measure $1$, can be
regarded as a similar inverse for the generalized Pitman transform
$\mathcal{P}$.\ We \emph{assume in the remaining of the paper that }%
$\overline{m}\in D_{\kappa}$. This permits to define a random walk
$\mathcal{W}$ and a Markov chain $\mathcal{H=P(W)}$ as in Section
\ref{Sec_WP}.\ Since $\overline{m}$ is fixed, we will denote for short by
$\mathbb{P}$ and $\mathbb{Q}$ the probability distributions $\mathbb{P}%
_{\overline{m}}$ and $\mathbb{Q}_{\overline{m}}$, respectively.

\subsection{Typical trajectories}

Consider $\overline{m}\in\mathcal{D}_{\kappa}$ and the associated
distributions $\mathbb{P}_{\overline{m}}$ and $\mathbb{Q}_{\overline{m}}$
defined on $\Omega$ and $\Omega_{\mathcal{C}}$, respectively. We introduce the
subsets of typical trajectories in $\Omega^{\mathrm{typ}},\Omega
^{\iota\mathrm{typ}}$ and $\Omega_{\mathcal{C}}^{\mathrm{typ}}$ as%
\begin{align*}
\Omega^{\mathrm{typ}} &  =\{\pi\in\Omega\mid\lim_{\ell\rightarrow+\infty}%
\frac{1}{\ell}\langle\pi(\ell),\alpha_{i}^{\vee}\rangle=\langle\overline
{m},\alpha_{i}^{\vee}\rangle\in\mathbb{R}_{>0}\quad\forall i=1,\ldots,n\},\\
\Omega^{\iota\mathrm{typ}} &  =\{\pi\in\Omega\mid\lim_{\ell\rightarrow+\infty
}\frac{1}{\ell}\langle\pi(\ell),\alpha_{i}^{\vee}\rangle=\langle
w_{0}(\overline{m}),\alpha_{i}^{\vee}\rangle\in\mathbb{R}_{<0}\quad\forall
i=1,\ldots,n\}\\
\Omega_{\mathcal{C}}^{\mathrm{typ}} &  =\{\pi\in\Omega_{\mathcal{C}}\mid
\lim_{\ell\rightarrow+\infty}\frac{1}{\ell}\langle\pi(\ell),\alpha_{i}^{\vee
}\rangle=\langle\overline{m},\alpha_{i}^{\vee}\rangle\in\mathbb{R}_{>0}%
\quad\forall i=1,\ldots,n\}.
\end{align*}
By Theorems \ref{ThLLN} and \ref{ThLLN P}, we have
\[
\mathbb{P}_{\overline{m}}(\Omega^{\mathrm{typ}})=1\text{ and }\mathbb{Q}%
_{\overline{m}}(\Omega_{\mathcal{C}}^{\mathrm{typ}})=1.
\]
We are going to see that the relevant Pitman inverse coincides with
$\mathcal{E}$ acting on the trajectories of $\Omega_{\mathcal{C}%
}^{\mathrm{typ}}$ and we will show that $\mathcal{E}(\mathcal{H})$ is then a
random trajectory with drift $w_{0}(\overline{m})$ where $w_{0}$ is the
longest element of the Weyl group ${\mathsf{W}}$.

\subsection{An involution on the trajectories\label{subsecLusztig}}

We have seen that the reverse map $r$ on paths defined in (\ref{fer}) flips
the actions of the operators $\tilde{e}_{i}$ and $\tilde{f}_{i}$ on any
connected crystal $B(\pi_{\kappa})$ of highest path $\pi_{\kappa}$%
.$\ $Nevertheless, we have
\[
r(B(\pi_{\kappa})\neq B(\pi_{\kappa})
\]
in general.\ So $r(\Omega)\neq\Omega$. To overcome this difficulty we can
replace our space of trajectories $\Omega$ by the set $\mathcal{L}_{\infty}$
of all infinite paths defined from the set $\mathcal{L}$ of \S \
\ref{subsecLit}. But $\mathcal{L}_{\infty}$ has not a probability space
structure neither a simple algebraic interpretation.\ Rather, it is
interesting to give another definition of $\mathcal{E}$ where the involution
$r$ is replaced by the Lusztig involution $\iota$ which stabilizes
$B(\pi_{\kappa})$ (see for example \cite{Len}).\ The longest element $w_{0}$
of the Weyl group $W$ (which is an involution) induces an involution $\ast$ on
the set of simple roots defined by $\alpha_{i^{\ast}}=-w_{0}(\alpha_{i})$ for
any $i=1,\ldots,n$.\ Write $\pi_{\kappa}^{low}$ for the lowest weight vertex
of $B(\pi_{\kappa})$, that is $\pi_{\kappa}^{low}$ is the unique vertex of
$B(\pi_{\kappa})$ such that $\tilde{f}_{i}(\pi_{\kappa}^{low})=0$ for any
$i=1,\ldots,n$.\ The involution $\iota$ is first defined on the crystal
$B(\pi_{\kappa})$ by
\[
\iota(\pi_{\kappa})=\pi_{\kappa}^{low}\text{ and }\iota(\tilde{f}_{i_{1}%
}\cdots\tilde{f}_{i_{r}}\pi_{\kappa})=\tilde{e}_{i_{1}^{\ast}}\cdots\tilde
{e}_{i_{r}^{\ast}}(\pi_{\kappa}^{low})
\]
for any sequence of crystal operators $\tilde{f}_{i_{1}},\ldots,\tilde
{f}_{i_{r}}$ with $r>0$. This means that $\iota$ flips the orientation of the
arrows of $B(\pi_{\kappa})$ and each label $i$ is changed in $i^{\ast}$. In
particular, we have $\mathrm{wt}(\iota(\pi))=w_{0}(\mathrm{wt}(\pi))$ for any
$\pi\in B(\pi_{\kappa})$.\ We extend $\iota$ by linearity on the linear
combinations of paths in $B(\pi_{\kappa})$.

\bigskip

\noindent We next define the involution $\iota$ on $B(\pi_{\kappa}%
)^{\otimes\ell}$ by setting
\[
\iota(\pi_{1}\otimes\cdots\otimes\pi_{\ell})=\iota(\pi_{\ell})\otimes
\cdots\otimes\iota(\pi_{1})
\]
for any $\pi_{1}\otimes\cdots\otimes\pi_{\ell}\in B(\pi_{\kappa})^{\otimes
\ell}$.\ It then follows from (\ref{Ten_Prod}) that for any any $i=1,\ldots,n$
we have
\begin{equation}
\iota\tilde{f}_{i}\iota(\pi_{1}\otimes\cdots\otimes\pi_{\ell})=\tilde
{e}_{i^{\ast}}(\pi_{1}\otimes\cdots\otimes\pi_{\ell}).\label{conjbeta}%
\end{equation}
Thus the involution $\iota$ flips the lowest and highest weight paths,
reverses the arrows and changes each label $i$ in $i^{\ast}$. In particular,
for any connected component $B(\eta)$ of $B(\pi_{\kappa})^{\otimes\ell}$, the
set $\iota(B(\eta))$ is also a connected component of $B(\pi_{\kappa
})^{\otimes\ell}$. In addition, we have%
\begin{equation}
\mathrm{wt}(\iota(\pi_{1}\otimes\cdots\otimes\pi_{\ell}))=w_{0}(\mathrm{wt}%
(\pi_{1}\otimes\cdots\otimes\pi_{\ell})).\label{iotaweight}%
\end{equation}

\begin{remark}
Observe that $\iota$ is very closed from $r$.\ The crucial difference is that
the crystals $\iota(B(\pi_{\kappa}))$ and $B(\pi_{\kappa})$ coincide whereas
$r(B(\pi_{\kappa}))$ is not a crystal in general.
\end{remark}

\begin{example}
We resume Example \ref{exam_typ} and consider $\mathfrak{g}=\mathfrak{sp}_{4}$
and $\kappa=\omega_{1}$. In this particular case we get $w_{0}=-id$ and
$\iota=r$ on $B(\pi_{\omega_{1}})$.\ We then have $\iota(\pi_{1}%
)=\pi_{\overline{1}}$ and $\iota(\pi_{2})=\pi_{\overline{2}}$. In the picture
below we have drawn the path $\eta$ and $\iota(\eta)$ where%
\begin{align*}
\eta &  =112111\bar{2}\bar{1}\bar{2}111222\bar{1}\bar{2}111222111\bar{2}%
\bar{1}22211,\\
\iota(\eta) &  =\bar{1}\bar{1}\bar{2}\bar{2}\bar{2}12\bar{1}\bar{1}\bar{1}%
\bar{2}\bar{2}\bar{2}\bar{1}\bar{1}\bar{1}21\bar{2}\bar{2}\bar{2}\bar{1}%
\bar{1}\bar{1}212\bar{1}\bar{1}\bar{1}\bar{2}\bar{1}\bar{1}.
\end{align*}
Here we simply write $a\in\{\bar{2},\bar{1},1,2\}$ instead of $\pi_{a}$ and
omitted for short the symbols $\otimes$.
\end{example}

\begin{center}%
\raisebox{-0cm}{\parbox[b]{10.1067cm}{\begin{center}
\includegraphics[
height=7.6091cm,
width=10.1067cm
]%
{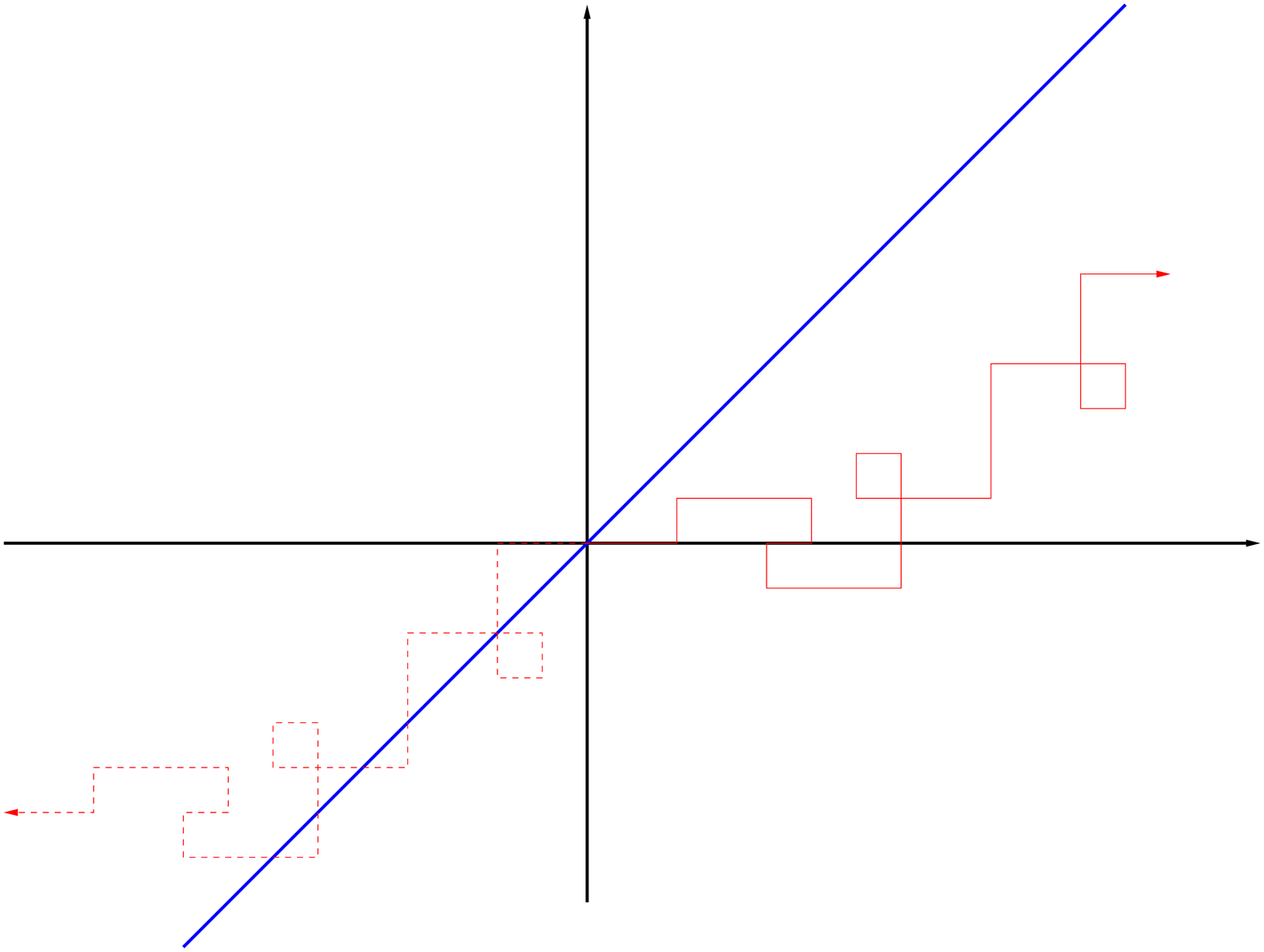}%
\\
The paths $\eta$ (in red) and $\beta(\eta)$ (in dashed read)
\end{center}}}%

\end{center}

The following proposition shows we can replace the involution $r$ by $\iota$
in the definition of the dual Pitman transform.

\begin{proposition}
\label{Lemma_rbeta}We have
\[
\mathcal{E}=r\mathcal{P}r=\iota\mathcal{P}\iota.
\]

\end{proposition}

\begin{proof}
Observe first that for any simple root $\alpha_{i}$, and any path $\eta\in
B(\pi_{\kappa})^{\otimes\ell}$, we have by (\ref{conjbeta}) $\mathcal{E}%
_{\alpha_{i}}(\eta)=\iota\mathcal{P}_{\alpha_{i^{\ast}}}\iota(\eta)$ because
the action of $\mathcal{E}_{\alpha_{i}}$ on any path reduces to a product of
operators $\tilde{f}_{i}$.\ Since $\mathcal{E}=\mathcal{E}_{\alpha_{1}}%
\cdots\mathcal{E}_{\alpha_{r}}$, we get $\mathcal{E}=\iota\mathcal{P}%
_{\alpha_{1^{\ast}}}\cdots\mathcal{P}_{\alpha_{r^{\ast}}}\iota$.\ But
$\mathcal{P}_{\alpha_{1^{\ast}}}\cdots\mathcal{P}_{\alpha_{r^{\ast}}%
}=\mathcal{P}_{\alpha_{1}}\cdots\mathcal{P}_{\alpha_{r}}=\mathcal{P}$ by
Proposition \ref{Prop-BBO} because $w_{0}=s_{\alpha_{1^{\ast}}}\cdots
s_{\alpha_{r^{\ast}}}$ is also a minimal length decomposition of $w_{0}$. We
therefore get $\mathcal{E}=\iota\mathcal{P}\iota$ as desired.
\end{proof}

\subsection{Dual random path}

{Let us define the probability distribution $p^{\iota}$ on $B(\pi_{\kappa})$
by setting
\begin{equation}
p_{\pi}^{\iota}=p_{\iota(\pi)}=\frac{\tau^{\kappa-w_{0}\mathrm{wt}(\pi)}%
}{S_{\kappa}(\tau)}\text{ for any }\pi\in B(\pi_{\kappa}) \label{def_p_iota}%
\end{equation}
and consider a random variable $Y$ with values in $B(\pi_{\kappa})$ and
probability distribution $p^{\iota}$. Set $m^{\iota}=E(Y),$ $\overline
{m}^{\iota}=m^{\iota}(1)$ and $\mathcal{D}_{\kappa}^{\iota}=w_{0}%
(\mathcal{D}_{\kappa})$.}

\begin{lemma}
We have

\begin{enumerate}
\item[(i)] $m^{\iota}=\iota(m)$

\item[(ii)] $\overline{m}^{\iota}=w_{0}(\overline{m})$. In particular,
$\overline{m}\in\mathcal{D}_{\kappa}$ if and only if $\overline{m}^{\iota}%
\in\mathcal{D}_{\kappa}^{\iota}$.
\end{enumerate}
\end{lemma}

\begin{proof}
By using that $\iota$ is an involution on $B(\pi_{\kappa})$, we get%
\[
m^{\iota}=\sum_{\pi\in B(\pi_{\kappa})}p_{\pi}^{\iota}\pi=\sum_{\pi\in
B(\pi_{\kappa})}p_{\iota(\pi)}\pi=\iota\left(  \sum_{\pi\in B(\pi_{\kappa}%
)}p_{\iota(\pi)}\iota(\pi)\right)  =\iota(m)
\]
which proves assertion (i). In particular, if we set $\overline{m}^{\iota
}=m^{\iota}(1)$, we have $\overline{m}^{\iota}=w_{0}(\overline{m})$ and
assertion 2 follows.
\end{proof}

\bigskip

{Similarly, we may consider the probability measure $(p^{\iota})^{\otimes\ell
}$ on $B(\pi_{\kappa})^{\otimes\ell}$ defined by
\[
(p^{\iota})^{\otimes\ell}(\pi_{1}\otimes\cdots\otimes\pi_{\ell})=p^{\iota}%
(\pi_{1})\cdots p^{\iota}(\pi_{\ell})=\frac{\tau^{\ell\kappa-w_{0}(\pi
_{1}(1)+\cdots\pi_{\ell}(1))}}{S_{\kappa}(\tau)^{\ell}}=\frac{\tau^{\ell
\kappa-w_{0}({\mathrm{wt}}(b))}}{S_{\kappa}(\tau)^{\ell}}.\quad^{(}%
\text{\footnote{We now have two probability measures $p^{\otimes\ell}$ and
$(p^{\iota})^{\otimes\ell}$ on $B(\pi_{\kappa})^{\otimes\ell}$. Observe that
for any event $E\subset B(\pi_{\kappa})^{\otimes\ell}$, we get
\begin{equation}
(p^{\iota})^{\otimes\ell}(E)=p^{\otimes\ell}(\iota(E)).\label{p_piota}%
\end{equation}
}}^{)}%
\]
}

By the Kolmogorov extension theorem, the family of probability measure
$((p^{\iota})^{\otimes\ell})_{\ell}$ admits a unique extension $\mathbb{P}%
^{\iota}:=(p^{\iota})^{\otimes{\mathbb{Z}_{\geq0}}}$ to the space
$B(\pi_{\kappa})^{\otimes\mathbb{Z}_{\geq0}}$. For any $\ell\geq1$, let
$Y_{\ell}:B(\pi_{\kappa})^{\otimes\mathbb{Z}_{\geq0}}\longrightarrow
B(\pi_{\kappa})$ be the canonical projection on the $\ell^{th}$ coordinate; by
construction, the variables $Y_{1},Y_{2},\cdots$ are independent and
identically distributed with the same law as $Y$. We denote by ${\mathcal{W}%
}^{\iota}$ the random path defined by
\[
{\mathcal{W}}^{\iota}(t):=Y_{1}(1)+Y_{2}(1)+\cdots+Y_{\ell-1}(1)+Y_{\ell
}(t-\ell+1){\text{ for }}t\in\lbrack\ell-1,\ell].
\]
Then ${\mathcal{W}}^{\iota}$ is defined on the probability space
$\Omega^{\iota}=(B(\pi_{\kappa})^{\otimes\mathbb{Z}_{\geq0}},\mathbb{P}%
^{\iota})$; notice that the set of trajectories of $\Omega^{\iota}$ is the
same as the one of $\Omega$ but that the probability $\mathbb{P}^{\iota}$ is
defined from $p^{\iota}$.\ We have in particular%
\[
\mathbb{P}^{\iota}(\Omega^{\iota\mathrm{typ}})=1.
\]
We also define the random walk $W^{\iota}=(W_{\ell}^{\iota})_{\ell\geq1}$ such
that $W_{\ell}^{\iota}=\mathcal{W}^{\iota}(\ell)$ for any $\ell\geq1$.\ Let
${\mathcal{H}}^{\iota}$ be the random process ${\mathcal{H}}^{\iota}%
=${$\mathcal{P}$}$({\mathcal{W}}^{\iota})$ and define $H^{\iota}=(H_{\ell
}^{\iota})_{\ell\geq1}$ such that $H_{\ell}^{\iota}=\mathcal{H}^{\iota}(\ell)$
for any $\ell\geq1$. We then have (see Proposition 4.6 in \cite{LLP3})

\begin{theorem}
\label{Th_LawH copy(1)} \ 

\begin{enumerate}
\item[(i)] For any $\beta,\eta\in P$, one gets
\[
{\mathbb{P}}^{\iota}(W_{\ell+1}^{\iota}=\beta\mid W_{\ell}^{\iota}%
=\eta)=K_{\kappa,\beta-\eta,}\frac{\tau^{\kappa-w_{0}(\beta-\eta)}}{S_{\kappa
}(\tau)}.
\]

\item[(ii)] The random sequence $H^{\iota}$ is a Markov chain with the same
law as $H$, that is with transition matrix%
\[
\Pi(\mu,\lambda)=\frac{S_{\lambda}(\tau)}{S_{\kappa}(\tau)S_{\mu}(\tau)}%
\tau^{\kappa+\mu-\lambda}m_{\mu,\kappa}^{\lambda}%
\]
where $\lambda,\mu\in P_{+}$.

\item[(iii)] For any path $\pi\in H_{\lambda}^{\ell}$,we have%
\[
\mathbb{P}^{\iota}({\mathcal{H}}^{\iota}=\pi)=\mathbb{P}({\mathcal{H}}%
=\pi)=\frac{\tau^{\ell\kappa-\lambda}S_{\lambda}(\tau)}{S_{\kappa}(\tau
)^{\ell}}.
\]

\end{enumerate}
\end{theorem}

\subsection{Asymptotic behavior in a fixed component}

Consider $\pi\in B(\pi_{\kappa})^{\otimes\ell}$ and $\eta\in\Omega$ such that
$\frac{1}{L}\langle\eta(L),\alpha_{i}^{\vee}\rangle$ converges to a positive
limit for any positive root $\alpha_{i},i=1,\ldots n$. For any $L,$ set
$\Pi_{L}(\eta)=\eta_{L}$ so that we have $\eta_{L}\in B(\pi_{\kappa})^{\otimes
L}$. Since $\pi\in B(\pi_{\kappa})^{\otimes\ell}$, the path $\eta_{L}%
\otimes\pi$ is defined on $[0,\ell+L].$ More precisely, we have $\eta
_{L}\otimes\pi(t)=\eta_{L}(t)$ for $t\in\lbrack0,L[$ and $\eta_{L}\otimes
\pi(t)=\eta_{L}(L)+\pi(t-L)$ for $t\in\lbrack L,\ell+L]$.

\begin{lemma}
\label{Lemma_Left}With the previous notation, we get%
\[
{\mathcal{P}}(\eta_{L}\otimes\pi)={\mathcal{P}}(\eta_{L})\otimes\pi
\]
for $L$ sufficiently large.
\end{lemma}

\begin{proof}
Recall that {$\mathcal{P}$}$=${$\mathcal{P}$}$_{\alpha_{i_{1}}}\cdots
${$\mathcal{P}$}$_{\alpha_{i_{r}}}$. One proves by induction that for any
$s=1,\ldots,r$, there exists a nonnegative integer $L_{s}$ such that
\[
{\mathcal{P}}_{\alpha_{i_{s}}}\cdots{\mathcal{P}}_{\alpha_{i_{r}}}(\eta
_{L}\otimes\pi)={\mathcal{P}}_{\alpha_{i_{s}}}\cdots{\mathcal{P}}%
_{\alpha_{i_{r}}}(\eta_{L})\otimes\pi
\]
for any $L>L_{s}$ and $\underset{L\rightarrow+\infty}{\lim}\langle
${$\mathcal{P}$}$_{\alpha_{i_{s}}}\cdots${$\mathcal{P}$}$_{\alpha_{i_{r}}%
}(\eta)(L),\alpha^{\vee}\rangle=+\infty$ for any simple root $\alpha$. The
lemma then follows by putting $s=1$.
\end{proof}

\bigskip

{Let $\mathcal{H}=(\mathcal{H}_{\ell})_{\ell\geq1}$ be a random process in
$\Omega_{\mathcal{C}}\subset\Omega$ with distribution $\mathbb{Q}%
_{\overline{m}}$.\ Since $\mathcal{H}$ takes value in $\Omega$, we can write
$\mathcal{H}_{\ell}=T_{1}\otimes\cdots\otimes T_{\ell}$ for any $\ell\geq1$,
where the random variable $T_{i}$ takes values in $B(\pi_{\kappa})$ for any
$i\geq1$. By Corollary \ref{Cor-Iso-ProbaSpace}, there exists a random process
$\mathcal{W}$ with values in $\Omega$ and distribution $\mathbb{P}%
_{\overline{m}}$ such that $\mathcal{H}$ and $\mathcal{P}(\mathcal{W})$
coincide $\mathbb{P}_{\overline{m}}$-almost surely. Notice that we also have
$\mathcal{W}_{\ell}=X_{1}\otimes\cdots\otimes X_{\ell}$ for any $\ell\geq1$,
where $X_{\ell}$ is a random variable with values in $B(\pi_{\kappa})$ with
the law defined in (\ref{defX}). }

\begin{proposition}
\label{prop_lim}$\mathbb{P}_{\overline{m}}$-almost surely, the random
variables $T_{\ell}$ and $X_{\ell}$ coincide for any large enough $\ell$.
\end{proposition}

\begin{proof}
Consider a trajectory $\omega\in\Omega^{\mathrm{typ}}$. For any $\ell\geq1$
and set $\Pi_{\ell}(\omega)=\pi_{1}\otimes\cdots\otimes\pi_{\ell}$. We can
apply Lemma \ref{Lemma_Left} to $\pi_{1}\otimes\cdots\otimes\pi_{\ell
-1}\otimes\pi_{\ell}$ since we have $\omega\in\Omega^{\mathrm{typ}}$.\ Hence,
for $\ell$ sufficiently large, we have
\[
\mathcal{P}(\pi_{1}\otimes\cdots\otimes\pi_{\ell-1}\otimes\pi_{\ell
})=\mathcal{P}(\pi_{1}\otimes\cdots\otimes\pi_{\ell-1})\otimes\pi_{\ell}.
\]
We thus have $\lim_{\ell\rightarrow+\infty}(T_{\ell}-X_{\ell})=0$ on
$\Omega^{\mathrm{typ}}$. We are done since $\mathbb{P}_{\overline{m}}%
(\Omega^{\mathrm{typ}})=1$.
\end{proof}

\subsection{The transformations $\mathcal{P}$ and $\mathcal{E}$ on infinite
paths}

The transformations $\mathcal{P}$ and $\mathcal{E}$ defined on $B(\pi_{\kappa
})^{\otimes\ell}$ can be extended to $\Omega$ and $\Omega_{\mathcal{C}%
}^{\mathrm{typ}}$, respectively. For any $\eta\in\Omega$ and any simple root
$\alpha$, set%
\[
{\mathcal{P}}_{\alpha}(\eta)(t)=\eta(t)-\inf_{s\in\lbrack0,t]}\langle
\eta(s),\alpha^{\vee}\rangle\alpha\quad\text{and}\quad{\mathcal{P}}%
(\eta)={\mathcal{P}}_{\alpha_{i_{1}}}\cdots{\mathcal{P}}_{\alpha_{i_{r}}}%
(\eta).
\]
Similarly, for any $\eta\in\Omega$ and any simple root $\alpha$ such that
$\lim_{t\rightarrow\infty}\langle\eta(t),\alpha^{\vee}\rangle=+\infty$, the
path ${\mathcal{E}}_{\alpha}(\eta)$ such that%
\[
{\mathcal{E}}_{\alpha}(\eta)(t)=\eta(t)-\inf_{s\in\lbrack t,+\infty\lbrack
}\langle\eta(s),\alpha^{\vee}\rangle\alpha+\inf_{s\in\lbrack0,+\infty\lbrack
}\langle\eta(s),\alpha^{\vee}\rangle\alpha
\]
for any $t\geq0$ is well defined.

\begin{proposition}
Consider $\eta$ in $\Omega_{\mathcal{C}}^{\mathrm{typ}}$.\ Then ${\mathcal{E}%
}(\eta)={\mathcal{E}}_{\alpha_{i_{1}}}\cdots{\mathcal{E}}_{\alpha_{i_{r}}%
}(\eta)$ is well defined and belongs to $\Omega^{\iota\mathrm{typ}}$.
\end{proposition}

\begin{proof}
We proceed by induction and show that ${\mathcal{E}}(\eta)={\mathcal{E}%
}_{\alpha_{i_{a}}}\cdots{\mathcal{E}}_{\alpha_{i_{r}}}(\eta)$ is well-defined
for any $a=1,\ldots,r$.\ It suffices to prove that
\[
\lim_{t\rightarrow\infty}\langle\eta(t),\alpha_{r}\rangle=+\infty\text{ and
}\lim_{t\rightarrow\infty}\langle{\mathcal{E}}_{\alpha_{i_{a+1}}}%
\cdots{\mathcal{E}}_{\alpha_{i_{r}}}\eta(t),\alpha_{a}\rangle=+\infty
\]
for any $a=1,\ldots r-1$. We get $\lim_{t\rightarrow\infty}\langle
\eta(t),\alpha_{r}\rangle=+\infty$ directly from the definition of
$\Omega_{\mathcal{C}}^{\mathrm{typ}}$.\ Now for any $a=1,\ldots,r-1$, and any
integer $\ell\geq0$, we have that ${\mathcal{E}}_{\alpha_{i_{a+1}}}%
\cdots{\mathcal{E}}_{\alpha_{i_{r}}}\eta(\ell)$ is the weight of the path
$\Pi_{\ell}(\eta)$.$\ $So we obtain by (\ref{PEHW})%
\[
\langle{\mathcal{E}}_{\alpha_{i_{a+1}}}\cdots{\mathcal{E}}_{\alpha_{i_{r}}%
}\eta(\ell),\alpha_{a}\rangle=\langle s_{i_{a+1}}\cdots s_{i_{r}}\eta
(\ell),\alpha_{a}\rangle=\langle\eta(\ell),s_{i_{r}}\cdots s_{i_{a+1}}%
(\alpha_{a})\rangle.
\]
Since $w_{0}$ is an involution, $w_{0}=s_{i_{r}}\cdots s_{i_{1}}$ is also a
minimal length decomposition. By (\ref{R+w0}), we know that $s_{i_{r}}\cdots
s_{i_{a+1}}(\alpha_{a})=\alpha$ is a positive root. It follows that
\[
\lim_{\ell\rightarrow\infty}\langle\eta(\ell),s_{i_{r}}\cdots s_{i_{a+1}%
}(\alpha_{a})\rangle=\lim_{\ell\rightarrow\infty}\langle{\mathcal{E}}%
_{\alpha_{i_{a+1}}}\cdots{\mathcal{E}}_{\alpha_{i_{r}}}\eta(\ell),\alpha
_{a}\rangle=+\infty.
\]
We finally get $\lim_{t\rightarrow\infty}\langle{\mathcal{E}}_{\alpha
_{i_{a+1}}}\cdots{\mathcal{E}}_{\alpha_{i_{r}}}\eta(t),\alpha_{a}%
\rangle=+\infty$ because
\[
\left\Vert {\mathcal{E}}_{\alpha_{i_{a+1}}}\cdots{\mathcal{E}}_{\alpha_{i_{r}%
}}(\eta(t))-{\mathcal{E}}_{\alpha_{i_{a+1}}}\cdots{\mathcal{E}}_{\alpha
_{i_{r}}}(\eta(\ell))\right\Vert \text{ with }\ell=\left\lfloor
t\right\rfloor
\]
is bounded by the common length of the elementary paths of $B(\pi_{\kappa})$,
uniformly in $\ell$. This proves that ${\mathcal{E}}(\eta)$ is
well-defined.\ Since $\eta\in\Omega_{\mathcal{C}}^{\mathrm{typ}}$, the path
$\eta_{\ell}=\Pi_{\ell}(\eta)$ is of highest weight.\ Thus, the path
$\mathcal{E}(\eta_{\ell})$ is of lowest weight.\ Comparing their weights, we
get $\mathcal{E}(\eta)(\ell)=w_{0}(\eta(\ell))$ which implies that
$\mathcal{E}(\eta)\in\Omega^{\iota\mathrm{typ}}$.
\end{proof}

\bigskip

Observe we have ${\mathcal{P}}(\eta)=\lim_{\ell\rightarrow+\infty}%
{\mathcal{P}}(\eta_{\ell})$ and ${\mathcal{E}}(\eta)=\lim_{\ell\rightarrow
+\infty}{\mathcal{E}}(\eta_{\ell})$ where $\eta_{\ell}=\Pi_{\ell}(\eta)$.

\subsection{Composition of the transformations $\mathcal{P}$ and $\mathcal{E}%
$}

Consider $\pi\in B(\pi_{\kappa})^{\otimes\ell},\eta\in\Omega_{\mathcal{C}%
}^{\mathrm{typ}}$ and $\xi\in\Omega^{\iota\mathrm{typ}}$. For any positive
integer $L$, set $\Pi_{L}(\eta)=\eta_{L}$ and $\Pi_{L}(\xi)=\xi_{L}$.

\begin{lemma}
\label{lemmaRight}With the above notation we have for $L$ sufficiently large

\begin{enumerate}
\item $\mathcal{PE}(\pi\otimes\eta_{L})=\pi\otimes\eta_{L}$ when $\pi
\otimes\eta_{L}$ is a highest weight path,

\item $\mathcal{EP}(\pi\otimes\xi_{L})=\pi\otimes\mathcal{E}(\xi_{L}).$
\end{enumerate}
\end{lemma}

\begin{proof}
1: Since $\pi\otimes\eta_{L}$ is a highest weight path, $\mathcal{E}%
(\pi\otimes\eta_{L})$ is the lowest weight path of $B(\pi\otimes\eta_{L})$,
the crystal associated to $\pi\otimes\eta_{L}$. Therefore $\mathcal{PE}%
(\pi\otimes\eta_{L})=\pi\otimes\eta_{L}$ is the highest weight path of
$B(\pi\otimes\eta_{L})$.

2: Since $\xi\in\Omega^{\iota\mathrm{typ}}$, we have for any $i=1,\ldots,n,$
$\underset{L\rightarrow+\infty}{\lim}\langle\xi_{L}(L),\alpha_{i}^{\vee
}\rangle=-\infty$. We get by (\ref{iotaweight})%
\[
\langle\iota(\xi_{L})(L),\alpha_{i}^{\vee}\rangle=\langle w_{0}(\xi
_{L}(L)),\alpha_{i}^{\vee}\rangle=\langle\xi_{L}(L),w_{0}(\alpha_{i}^{\vee
})\rangle=-\langle\xi_{L}(L),\alpha_{i^{\ast}}^{\vee}\rangle
\]
for any $i=1,\ldots,n$. So $\underset{L\rightarrow+\infty}{\lim}\langle
\iota(\xi_{L})(L),\alpha_{i}^{\vee}\rangle=+\infty$ for any $i=1,\ldots,n$.
Recall the $\iota\mathcal{P}=\mathcal{E}\iota$ and $\iota\mathcal{E}%
=\mathcal{P}\iota$ by Lemma \ref{Lemma_rbeta}. We have the equivalencies
\[
\mathcal{E}(\pi\otimes\xi_{L})=\pi\otimes\mathcal{E}(\xi_{L}%
)\Longleftrightarrow\iota\mathcal{E}(\pi\otimes\xi_{L})=\iota(\pi
\otimes\mathcal{E}(\xi_{L}))\Longleftrightarrow\mathcal{P}(\iota(\xi
_{L})\otimes\iota(\pi))=\mathcal{P}(\iota(\xi_{L}))\otimes\iota(\pi).
\]
But the last equality hold by Lemma \ref{Lemma_Left} for $L$ sufficiently
large. This proves that $\mathcal{E}(\pi\otimes\xi_{L})=\pi\otimes
\mathcal{E}(\xi_{L})$ for $L$ sufficiently large. Now, observe that
$\pi\otimes\xi_{L}$ and $\mathcal{E}(\pi\otimes\xi_{L})=\pi\otimes
\mathcal{E}(\xi_{L})$ both belong to the crystal $B(\pi\otimes\xi_{L})$.\ In
this crystal the transforms $\mathcal{P}$ and $\mathcal{E}$ return the highest
and lowest paths, respectively. Therefore, we have $\mathcal{EP}(\pi\otimes
\xi_{L})=\mathcal{EP}(\pi\otimes\mathcal{E}(\xi_{L}))$.\ But $\pi
\otimes\mathcal{E}(\xi_{L})=\mathcal{E}(\pi\otimes\xi_{L})$ is the lowest path
of $B(\pi\otimes\xi_{L})$. This implies that $\mathcal{EP}(\pi\otimes\xi
_{L})=\pi\otimes\mathcal{E}(\xi_{L})$ for $L$ sufficiently large as desired.
\end{proof}

\begin{theorem}
\ \label{Th_PE}

\begin{enumerate}
\item For any $\eta\in\Omega_{\mathcal{C}}^{\mathrm{typ}}$, we have
$\mathcal{PE}(\eta)=\eta$.

\item For any $\xi\in\Omega^{\iota\mathrm{typ}}$, we have $\mathcal{EP}%
(\xi)=\xi$.
\end{enumerate}
\end{theorem}

\begin{proof}
Consider $\ell$ a positive integer.\ For any integer $L\geq\ell$ we can write
$\Pi_{L}(\eta)=\Pi_{\ell}(\eta)\otimes\eta_{L}$ and $\Pi_{L}(\xi)=\Pi_{\ell
}(\xi)\otimes\xi_{L}$ with $\eta_{L}$ and $\xi_{L}$ in $B(\pi_{\kappa
})^{\otimes L-\ell}$. Since $\eta\in\Omega_{\mathcal{C}}^{\mathrm{typ}}$ and
$\xi\in\Omega^{\iota\mathrm{typ}}$, we have for any simple root $\alpha_{i}$,
\[
\underset{L\rightarrow+\infty}{\lim}\langle\eta_{L}(L),\alpha_{i}^{\vee
}\rangle=+\infty\text{ and }\underset{L\rightarrow+\infty}{\lim}\langle\xi
_{L}(L),\alpha_{i}^{\vee}\rangle=-\infty.
\]
So by applying Lemma \ref{lemmaRight}, we get for $L$ sufficiently large
(depending on $\ell$)%
\[
\mathcal{PE}(\Pi_{L}(\eta))=\Pi_{\ell}(\eta)\otimes\eta_{L}\text{ and
}\mathcal{EP}(\Pi_{L}(\xi))=\Pi_{\ell}(\xi)\otimes\mathcal{E(}\xi_{L})
\]
for any $\ell\leq L$.\ This shows that $\mathcal{PE}(\eta)=\eta$ and
$\mathcal{EP}(\xi)=\xi$ by taking the limit when $\ell$ tends to infinity.
\end{proof}

\begin{remark}
It is possible to state a slightly stronger statement of the previous theorem
where $\Omega$ is replaced by $\mathcal{L}_{\infty}$ (see \S \
\ref{subsecLusztig}) in the definition of $\Omega_{\mathcal{C}}^{\mathrm{typ}%
}$ and $\Omega^{\iota\mathrm{typ}}$.
\end{remark}

Write $\mathcal{W}^{\iota}=Y_{1}\otimes Y_{2}\cdots$ the dual random path with
drift $\iota(\overline{m})$. The following Theorem shows that the
transformation $\mathcal{E}$ defined on $\Omega_{\mathcal{C}}^{\mathrm{typ}}$
can be regarded as the inverse of the generalized Pitman transform
$\mathcal{P}$. Recall that for both random trajectories $\mathcal{W}^{\iota}$
and $\mathcal{W}$, we have $\mathcal{H}=\mathcal{P}(\mathcal{W})=\mathcal{P}%
(\mathcal{W}^{\iota})$.

\pagebreak

\begin{theorem}
Assume $\overline{m}\in\mathcal{D}_{\kappa}$. Then we have

\begin{enumerate}
\item[(i)] $\mathcal{EP(W}^{\iota})=\mathcal{W}^{\iota}$ $\mathbb{P}^{\iota}%
$-almost surely,

\item[(ii)] We have $\mathcal{E}(\mathcal{H})=Y_{1}\otimes Y_{2}\otimes\cdots$
where the sequence of random variable $(Y_{\ell})_{\ell\geq1}$ is i.i.d. and
each variable $Y_{\ell},\ell\geq1$ has law $Y$ as defined in (\ref{def_p_iota}).

\item[(iii)] $\mathcal{PE}(\mathcal{H)}=\mathcal{H}$ $\mathbb{Q}$-almost surely.
\end{enumerate}
\end{theorem}

\begin{proof}
(i) Write $\mathcal{W}^{\iota}=Y_{1}\otimes Y_{2}\cdots.\ $Since
$\mathbb{P}^{\iota}(\Omega^{\iota\mathrm{typ}})=1$, we get $\mathcal{EP(W}%
^{\iota})=\mathcal{W}^{\iota}$ $\mathbb{P}^{\iota}$-almost surely by Assertion
2 of Theorem \ref{Th_PE}. Since $\mathcal{P}(\mathcal{W}^{\iota}\mathcal{)=H}%
$, we have $\mathcal{E}(\mathcal{H})=\mathcal{EP(W}^{\iota})$.\ By assertion
(i), this means that $\mathcal{E}(\mathcal{H})=\mathcal{W}^{\iota}$ which
proves assertion (ii).

To obtain assertion (iii), it suffices to observe that $\mathcal{PE}%
(\mathcal{H)}=\mathcal{H}$ $\mathbb{Q}$-almost surely by Assertion 1 of
Theorem \ref{Th_PE} since we have $\mathbb{Q}(\Omega_{\mathcal{C}%
}^{\mathrm{typ}})=1$.
\end{proof}

\bigskip

\noindent{cedric.lecouvey@lmpt.univ-tours.fr}\newline%
{emmanuel.lesigne@lmpt.univ-tours.fr}\newline{marc.peigne@lmpt.univ-tours.fr}

\end{document}